\documentclass[12pt]{article}
\usepackage{amsmath,amssymb}
\usepackage{amsthm}

\def\AC{\mathcal{A}}
\def\PC{\mathcal{P}}

\def\MC{\mathcal{M}}

\def\E{\mathbf{E}}

\def\N{\mathbf{N}}

\def\R{\mathbf{R}}

\def\Z{\mathbf{Z}}

\def\x{\mathbf{x}}

\def\1{\mathbf{1}}

\def\u{\mathbf{u}}

\def\al{\alpha}
\def\be{\beta}
\def\pa{\partial}
\def\ep{\epsilon}
\def\de{\delta}
\def\ga{\gamma}
\def\ka{\varkappa}

\newtheorem{theorem}{Theorem}[section]
\newtheorem{lemma}{Lemma}[section]

\newtheorem{remark}{Remark}

\newcommand{\la}{\lambda}

\newcommand{\om}{\omega}
\newcommand{\Ga}{\Gamma}

\newcommand{\Si}{\Sigma}
\newcommand{\Om}{\Omega}
\newcommand{\La}{\Lambda}

\begin{document}
\title{The evolutionary game of pressure (or interference), resistance and collaboration
\thanks{http://arxiv.org/abs/1412.1269}}
\author{
Vassili N. Kolokoltsov\thanks{Department of Statistics, University of Warwick,
 Coventry CV4 7AL UK,
  Email: v.kolokoltsov@warwick.ac.uk and associate member  of IPI RAN RF}}
\maketitle

\begin{abstract}
In this paper we extend the framework of evolutionary inspection game put forward recently by the author and coworkers
to a large class of conflict interactions dealing with the pressure executed by the major player
(or principal) on the large group of small players
that can resist this pressure or collaborate with the major player.
We prove rigorous results on the convergence of various Markov decision models of interacting small agents
(including evolutionary growth),
namely pairwise, in groups and by coalition formation,
to a deterministic evolution on the distributions of the state spaces of small players paying main attention
to situations with an infinite state-space of small players.
We supply rather precise rates of convergence.
The theoretical results of the paper are applied to the analysis
of the processes of inspection, corruption, cyber-security, counter-terrorism,
banks and firms merging, strategically enhanced preferential attachment and many other.
\end{abstract}


\section{Introduction}

\subsection{Objectives and content of the study}

The inspection games represent an important class of games with various
applications from the arms race control to the study of tax evasion,
see  e. g. \cite{AvSZ2002} for a general survey, as well as
\cite{A2008}, \cite{AC2005}, \cite{AK2004} and references therein.
In \cite{KoPaYa} the author with coworkers initiated the study of
the inspection games from the evolutionary perspective, aimed at analysis
of the class of games with large number of inspectees.

The aim of the present paper is two-folds: 1) To widen the range of applicability
of this research by introducing a unified methodology for the analysis of a large class of conflict
interactions of social, economic or military character (that turn out to be mathematically similar,
but are often discussed in disjoint sets of subject specific journals)
describing the pressure executed by a big player (or principal)
on a large group of small players that resist the pressure or collaborate, that is the class
of games of an agent immersed into a pool of evolutionary and mean-field interacting
small players; 2) to build the rigorous mathematical theory of the law of large number limits
for the latter conflicts by proving that the controlled deterministic
evolutionary equation (kinetic equation) describing the dynamics of interaction can be obtained as the
limiting behavior of the controlled Markov models of $k$th order and/or mean-field
interaction (with the number of agents tending to infinity) and thus extending the corresponding theory
for the justification of the usual replicator dynamics (see e.g. \cite{BeWe}
or Section 11.9 of textbook \cite{KolMal1} for the latter).
The practical usefulness of this limit is that it provides much more tractable limiting
models where carrying out a traditional Markov decision analysis for a large state space is often unfeasible.

The paper is organized as follows.
In the next introductory subsections we discuss the related literature on the dynamic law of large numbers
and then  motivate our analysis by invoking certain real life conflict interactions
that can be analyzed via our general model providing social, economic, historic, geopolitical
and literary perspectives. The next section is devoted to the simple case of a 'short-sighted'
principal with the direct best response strategy.
We deduce rather precise convergence rates in terms of the averages of smooth functions (rather than
more developed  estimates for trajectories, see \cite{BeWe}) and provide the crucial
link between the fixed point of the limiting dynamics and the Nash equilibria
of the corresponding $N$-player game (which is quite different from the usually discussed link with
the underlying two player game of the standard evolutionary setting, which is not defined in our setting).
This simplest framework presents a handy opportunity to discuss in the most transparent way
our basic examples of payoffs related to various contexts thus leading to a unified theory
of various subject areas. Subsection \ref{secoptimalocandgroup} is devoted to more-or-less straightforward
(from the mathematical point of view) extensions of the basic model that include
the possibility of simultaneous interactions of more than two players ($k$th order interaction),
as well as of diversified strategies of the principle solving the
optimal allocation problem on evolutionary background.
Section \ref{secforwardprin1} provides the convergence (with the rates) of $N$-player games
to a deterministic limit for a more sophisticated (but more realistic) setting of a forward looking
major player. Section \ref{secmodgrowthpres} initiates the analysis of the controlled law of large numbers for processes
with unbounded intensities defined on a countable (rather than finite or compact) state space, which leads to modeling
processes of evolutionary growth with variable population
size of small players. This includes various processes of birth, death, migration and coalition formation,
which are strategically enhanced in the sense that their evolutions are subject to controlled external pressure.
In Appendix we explain some auxiliary facts about variational derivatives,
ODEs in Banach spaces and the comparison of semigroups.

%The mathematical technique used in the paper is based on the theory
%of nonlinear Markov processes developed in \cite{Ko10}, see also \cite{Ko04} and \cite{Ko06}.

Let us indicate further steps (in addition to those outlined in Subsection \ref{secbirthdeath}
and at the ends of most of the sections)
that are worth being exploited in the future work on the models discussed here.
(1) It should be of interest to analyze next order approximation to the dynamic law of large numbers studied here,
which can be carried out in two similar (but different) ways: by including in the generator the second order (diffusive)
terms of order $1/N$ (as is done in paper \cite{MGT} for standard evolutionary games or in \cite{GoShah11},
\cite{GoKo14} for the chemical kinetics setting) or by systematic study of fluctuations
as dynamic central limit theorem (as done in \cite{Ko10} for classical models). (2) It is natural to include
possible spatial distributions (which can lead to quite remarkable effects, see e.g.  \cite{YNRS}) aiming at
the analysis of various models of crime detection and relating to the well developed theory of patrolling games,
see \cite{Alp14}, \cite{Alp13}, \cite{Alp11} and references therein.
(3) We consider a single major player in the pool of small players; it is natural to extend the model to
the general finite player game on the evolutionary background.
(4) Allowing the principal to withdraw from the interaction (to retire) would lead to the optimal
 stopping problem on the evolutionary background and, in particular, to the evolutionary extension
 of the well studied multi-armed bandit problem (see e. g. \cite{ElKaKa} and \cite{GinBandit89}
 for the background on the latter).

 All notations for the norms and spaces used are carefully introduced in the Appendix.

\subsection{Related work on dynamic law of large numbers}

In this section we discuss the papers
that are relevant more to our methodology itself rather than its concrete applications.
Roughly speaking, this methodology concerns the rigorous derivation of the dynamic law of large numbers
for Markov dynamics with control, competition and/or cooperation.
The literature on the topic is quite abundant and keeps growing rapidly.

First of all, our model of evolutionary type behavior of species in reaction to the
actions of the distinguished major player bears similarity
with the recently developed models of mean-field games with a major
player (see \cite{Hu10}, \cite{NoCa13}, \cite{WaZh13}, \cite{KoYaInsp}),
where also the necessity to consider
various classes of players is well recognized, see also
\cite{Bens} and \cite{CarDel}.
However, unlike the mean-field game setting,
(see e. g. \cite{BenFr}, \cite{LL2006}, \cite{HCM3}), our species do not
rationally optimize the strategies based on the observed environment,
but rather mechanically copy (myopic hypothesis) better strategies of randomly chosen neighbors.

The paper \cite{GaGaLe} proves the convergence (after a natural scaling) of a centrally
controlled discrete-time Markov chain of large number of constituents
to the deterministic continuous-time dynamics given by ordinary differential equations.
Similar results are obtained in \cite{Ko12} for continuous-time Markov chains
with possibly competitive control.

The derivation of various evolutionary dynamics as the dynamic law of large number for
 Markov models of binary or mean-field interaction is well developed in the literature
 on evolutionary games. For instance, paper \cite{Boy95} proves the convergence to
 a deterministic ODE of the Markov model, where the pairwise interaction is organized in
 discrete time so that at any moment a given fraction $\al(N)$ of a homogeneous population of $N$ species
 is randomly chosen and decomposed into matching pairs, which afterwards experience simultaneous
transformations into other pairs according to a given distribution. Paper
\cite{CorSar00} extends this setting to include several types of species
and the possibility of different scaling that may lead, in the limit $N\to \infty$, not only to ODE,
but to a diffusion process. In \cite{Ko04a} the general class of stochastic dynamic law of large number
is obtained from binary or more general $k$th order interacting particle systems
(including jump-type and L\'evy processes as a noise). The study of \cite{BeWe}
 concentrates on various subtle estimates for the deviation
of the limiting deterministic evolution from the approximating Markov chain for the evolution
that allows a single player (at any random time) to change her strategy to the strategy of another
randomly chosen player.

%A slightly different (but still very close) trend of research represents
%the analysis of general stochastic approximation in association with the so-called method of ordinary differential equations,
%see e. g. \cite{RoSa13} and \cite{Be96} and references therein.

 A related trend of research analyzes various choices of Markov approximation to repeated games and their consequences
to the question of choosing a particular Nash equilibrium amongst the usual multitude of them.
Seminal contribution \cite{KaMaRo} distinguishes specifically the myopic hypothesis, the mutation or experimentation
hypothesis and the inertia hypothesis in building a Markov dynamics of interaction.
As shown in \cite{KaMaRo} (with similar result in \cite{You93}), introducing
mutation of strength $\la$ and then passing to the limit $\la \to 0$ allows one to choose a certain particular Nash equilibrium,
called a long run equilibrium (or statistically stable, in the terminology of \cite{FoYou}) that for some coordination games turns out to coincide
with the risk-dominant (in the sense of \cite{HaSe}) equilibrium. Further important contributions in this direction include
\cite{Elli}, \cite{BinSam}, \cite{BinSamVa} showing how different equilibria could be obtained by a proper
fiddling with noise (for instance local or uniform as in \cite{Elli}) and discussing the important practical
 question of 'how long' is the 'long-run' (for a recent progress on this question see \cite{KrYi13}).
In particular paper \cite{BinSamVa} discusses in detail the crucial question of the effect of applying
the limits $t\to \infty$, $\tau \to 0$ (the limit from discrete to continuous replicator dynamics), $N\to \infty$ and
$\la \to 0$ in various order. Further development of the idea of local interaction leads naturally
to the analysis of the corresponding Markov processes on large networks, see \cite{LoPi06} and references therein.
Some recent general results of the link between Markov approximation to the mean field (or fluid) limit can be found in
\cite{LeBo13} and \cite{BeLeBo08}.
Though in many papers on Markov approximation, the switching probabilities of a revising player depends
on the current distribution of strategies used (assuming implicitly that this distribution is observed by all players)
there exist also interesting results (initiated in \cite{San01}, see new developments in \cite{San12})
arising from the assumption that the switching of a revising player is based
on an observed sample of given size of randomly chosen other payers.

In the abundant literature on the models of evolutionary growth (see \cite{SimRoy} for a review),
the discussion usually starts directly with the deterministic limiting model,
with the underlying Markov model being just mentioned as a motivating heuristics.

\subsection{Informal description of the model}
\label{secinformprinc}

The models we discuss here in laymen terms will be given precise mathematical meaning in Subsection
\ref{secexampmatpressureevol}.

In the inspection game with a large number of inspectees, see \cite{KoPaYa}, any one from a large group of $N$ inspectees
has a number of strategies parametrized by a finite or infinite set of nonnegative numbers $r$ indicating the level at which she
chooses to break the regulations ($r=0$ corresponds to the full compliance). These can be the levels of tax evasion, the levels
of illegal traffic through a check point, the amounts at which the arms production exceeds the agreed level, etc. On the other hand, a specific player,
the inspector, tries to identify and punish the trespassers. Inspector's strategies are
real numbers $b$ indicating the level of her involvement in the search process, for instance, the budget spent on it, which
is related in a monotonic way to the probability of the discovery of the illegal behavior of trespassers.
The payoff of an inspectee depends on whether her illegal behavior is detected or not. If social norms are taken into account,
 this payoff will also depend on the overall crime level
 of the population, that is, on the probability distribution of inspectees playing different strategies.
 The payoff of the inspector may depend on the fines collected from detected violators, on the budget spent and again
on the overall crime level (that she may have to report to governmental bodies, say).  As time goes by, random pairs of inspectees can communicate
in such a way that one inspectee of the pair can start copying the strategy of another one if it turns out to be more beneficial.
Then one can argue that this evolution (or more precisely, its limit as $N\to \infty$)
eventually settles down to one of its stable equilibria. The analysis of such equilibria was the main objective of \cite{KoPaYa}.

This model naturally extends to a more general setting where a distinguished 'big' player exerts certain level $b$ of pressure
on (or interference into the affairs of) a large group of $N$ 'small' players that can resist this pressure on a level $r$.
The term 'small' reflects the idea that the influence of each particular player becomes negligible as $N\to \infty$.
As an example of this general setting one can mention the interference of humans on the environment (say, by hunting or fishing)
or the use of medications to fight with infectious bacteria in a human body,
with resisting species having the choice of occupying the areas of ample foraging but more dangerous interaction with the big player
(large resistance levels $r$) or less beneficial but also less dangerous areas (low $r$). Another example can be the level of
 resistance of the population on a territory occupied by military forces.

 A slightly new twist to the model presents
 the whole class of games modeling corruption (see \cite{Aidt}, \cite{Jain},
 \cite{LaMoMaRa09}, \cite{Mal14} and \cite{KolMa15} and references therein for a general background). For instance,
 developing the initial simple model of \cite{Beck}, a large class of these games studies the strategies
 of a benevolent principal (representing, say, a governmental body that is interested in the efficient development
 of economics) that delegates a decision-making power to a non-benevolent (possibly corrupt) agent, whose behavior (legal or not)
 depends on the incentives designed by the principal. The agent can deal, for example, with tax collection of firms.
The firms can use bribes to persuade a corrupted tax collector to accept falsified revenue reports.
In this model the set of inspectors can be considered as a large group of small players
 that can choose the level of corruption (quite in contrast to the classical model of inspection)
 by taking no bribes at all, or not too much bribes, etc.
The strategy of the principal consists in fiddling with two instruments:
choosing wages for inspectors (to be attractive enough, so that the agents should be afraid to loose it) and investing
in activities aimed at the timely detection of the fraudulent behavior. Mathematically these two
 types are fully analogous to preemptive and defensive methods discussed in the literature on counterterrorism
 (described in detail below in Subsection \ref{secexampmatpressureevol}).

Another 'linguistic twist' that changes 'detected agents' to 'infected agents' brings us directly
to the (seemingly quite different) setting of cyber-security or biological attack-defence games. Yet another
'turn of the screw' that extends the setting (more-or-less straightforwardly) to possibly different classes of small players,
brings us to the domain of optimal allocation games, but now in the competitive evolutionary setting, where
the principal (say an inspector) has the task to distribute limited resources as efficiently as possible.
As another related area let us stress the analysis of terrorism and counterterrorist measures, where it is natural
to consider terrorists or terrorists organizations as small players against a principal representing a government
of a target country.

Furthermore, in many situations, the members of the pool of small players have an alternative class of strategies
of collaborating with the big player on various levels $c$. The creation of such possibilities can be considered as a
strategic action of the major player (who can thus exert some control on the rules of the game).
In biological setting this is, for instance, the strategy of dogs
joining humans in hunting their 'relatives' wolves or foxes (nicely described poetically as the talk between a dog and a fox
in the famous novel \cite{SBambi}). Historical examples include the strategy of slaves helping their masters to terrorize
and torture other slaves and by doing this gaining for themselves more beneficial conditions, as described e.g. in the classics \cite{TomCab}.
As a military example one can indicate the strategy of the part of the population on a territory occupied by foreign militaries
that joins the local support forces for the occupants, for US troops in Iraq this strategy being well discussed in
Chapter 2 of \cite{Mesq}. Alternatively, this is also the strategy of population helping police to fight with criminals and/or terrorists.
In the world of organized crime it is also a well known strategy to play simultaneously both resistance (committing crime) and
collaboration (to collaborate with police to get rid of the competitors), the classic presentation in fiction being novel \cite{Fieldi}.

It is worth stressing the existence of a large number of problems, where it is essential to work with infinite state-space
of small players, in particular, with the state-space being the set of all natural numbers. Mathematical results are much rare
for this case, as compared with finite state-spaces, and we pay much attention to it. This infinite-dimensional setting is crucial
for the  analysis of models with growth, like merging banks or firms on the market (see \cite{Pushkin04} and \cite{SaMaSo})
or the evolution of species and the development of networks with preferential attachment (the term coined in \cite{BaAl99}),
for instance scientific citation networks or the network of internet links (see a detailed discussion in \cite{KraRed}).
Models of growth are known to lead to power laws in equilibrium, which are verified in a variety of real life processes,
see e.g. \cite{SaMaSo} for a general overview and \cite{Rich48} for particular applications in crime rates.
Here we are interested in the response of such system to external parameters that may be set by the principal
(say, by governmental regulations) who has her own agenda (may wish to influence the growth of certain economics sectors).
Apart from the obvious economic examples mentioned above, similar process of the growth of coalitions under pressure
can be possibly used for modeling the development of human cooperation (forming coalitions
under the 'pressure' exerted by the nature) or the creation of liberation armies
(from the initially small guerillas groups) by the population of the territories oppressed by an external military force.
Of course these processes have a clear physical analogs, say the formation of dimers and trimers by the molecules of gas
with eventual condensation under (now real physical) pressure. The relation with the Bose-Einstein condensation is also
well known, see e. g. \cite{BiBa01} and \cite{SimRoy}.

\section{The best response principal}

\subsection{Discrete setting}
\label{secbestrespprin1}

We shall consider a game of a major 'big' player $P$ (the principal) with a group of small (indistinguishable) players.
The strategies of the big player are points $b$ in a compact convex subset of a Euclidean space. In the simplest
examples points $b$ belong to a closed interval and can be interpreted as
 the level of involvement in the actions of the group (say, a budget of a big player).
In general, its multidimensional character is natural as describing possible various
instruments that can be used to influence other players or various allocations to groups of small players with various strategies.

Let us start with the case of a finite number of strategies $\{1, \cdots, d\}$ of each small player.
Thus the state space of the group is $\Z^d_+$, the set of sequences of $d$ non-negative integers $n=(n_1,...,n_d)$,
 where each $n_i$ specifies the number of players in the state $i$. Let $N$ denote the total
 number of players in the state $n$: $N=n_1+...+n_d$. For $i\neq j$ and a state $n$ with $n_i>0$ denote by
 $n^{ij}$ the state obtained from $n$ by removing one agent of type $i$ and adding an agent of type $j$,
 that is $n_i$ and $n_j$ are changed to $n_i-1$ and $n_j +1$ respectively.
Let the payoff $R_i(x,b)$ of the strategy $i$ against the player $P$ be a continuous function of the strategy $b$ of $P$ and
the overall distribution
\[
x=(x_1, \cdots, x_d)= (n_1, \cdots, n_d)/N \in \Si_d
\]
of the strategies applied, where $\Si_d$ is the standard simplex of vectors with non-negative coordinates
summing up to $1$ (that is, the set of probability laws on $\{1, \cdots, d\}$).

Assuming that $P$ has some strategy $b(x,N)$ let us consider the following Markov model of the interaction of the group.
With some rate $\ka/N$ any pair of agents can meet and discuss their payoffs. This discussion may result in the player with lesser
payoff $R_i$ switching to the strategy with the better payoff $R_j$, which may occur with probability $\al(R_j-R_i)$,
where $\al>0$ is a proportionality constant. In future we set $\al=1$, as it can be directly incorporated in $\ka$.

\begin{remark}
We are working here with a pure myopic behavior for simplicity. Introduction of random mutation
on global or local levels (see e. g. \cite{KaMaRo} for standard evolutionary games) would
not affect essentially the convergence result below,
but would lead to serious changes in the long run of the game, which are worth being exploited.
\end{remark}

More rigorously, the process is described as follows. At the initial moment to any pair of agents $\{A_i,A_j\}$
(where $A_i$ and $A_j$ are in the state $i$ and $j$ respectively) is attached a random clock,
which will click after $\al |R_j-R_i|/N$-exponential waiting time (the expectation of this time is $N/\al |R_j-R_i|$).
The minimum of all these independent
$N(N-1)$ exponential waiting times is of course
also an exponential waiting time. If this minimum is realized on the pair $\{A_i,A_j\}$, then the agent with the lower $R$, say $A_i$,
changes her state to the one with higher $R$, say $A_j$, and the process continues analogously from the new state
(all clocks are set to zero). (Alternatively, the same process is described by one exponential clock such that, when it clicks,
the updating pair $(i,j)$ is chosen with probability proportional to the product $n_in_j$ of sizes of each strategy
and the difference of their payoffs.)
This process is a continuous-time Markov chain on $\Z^d_+$ with the generator
\[
L_{b,N}f (n)=\frac{1}{N}\sum_{i,j: R_j(n/N,b(n/N,N))>R_i(n/N,b(n/N,N))} \ka n_i n_j
\]
\begin{equation}
\label{eqdefgenmeanfpool1}
\times
[R_j(n/N,b(n/N,N))-R_i(n/N,b(n/N,N))][f(n^{ij})-f(n)].
\end{equation}
In terms of distributions $x=n/N$ it becomes
\[
L_{b,N}f (x)=N\sum_{i,j: R_j(x,b(x,N))>R_i(x,b(x,N))} \ka x_i x_j
\]
\begin{equation}
\label{eqdefgenmeanfpool2}
\times
[R_j(x,b(x,N))-R_i(x,b(x,N))][f(x-e_i/N+e_j/N)-f(x)],
\end{equation}
where $e_1,...,e_d$ denotes the standard basis in $\R^d$.

We are interested in the asymptotic behavior of the chains
generated by $L_{b,N}$, as $N \to \infty$.
As will be shown, the limiting process turns out to be a deterministic one governed by the system of ODE
\begin{equation}
\label{eqdefgenmeanfpool6}
\dot x_j=\sum_{i} \ka x_i x_j[R_j(x,b(x))-R_i(x,b(x))], \quad j=1,...,d,
\end{equation}
which is the system of kinetic equations generalizing (and modifying) the usual replicator
dynamics. At the end of this section we shall discuss some consequences to the corresponding game with
finite number of players.

\begin{remark}
The heuristic arguments leading to the equations of type \eqref{eqdefgenmeanfpool6} are well presented
in the literature (see e. g. \cite{B2004} or \cite{KoPaYa}) and will not be reproduced here.
The general context of deterministic limit is discussed in \cite{Ko12}.
\end{remark}

To go further we have to model the behavior of the major player.
As a warm-up, we start in this section with a simpler case of a short-sighted major player
that can make instantaneous adjustments to her strategy without additional costs.
Namely, let us assume that the payoff of $P$ playing against the group of small players is given by
a function $B(x,b,N)$, which is smooth and concave in $b$, so that for all $x,N$
the maximum point
\begin{equation}
\label{eqdefgenmeanfpool6b0}
b^*(x,N)=argmax \, B(x,b,N)
\end{equation}
is uniquely defined, and that $P$ chooses $b^*(x,N)$ as her strategy at any time.

Let us denote by $X^*_{N}(t,x)$ the Markov chain generated by \eqref{eqdefgenmeanfpool2}
and starting in $x\in \Z^d_+/N$ at the initial time $t=0$,
with $b^*$ used instead of $b$.

We use the (standard) notations for norms, Lipschitz norms and functional spaces specified in Appendix \ref{secnotspace}.

\begin{theorem}
\label{th1}
 Assume
\begin{equation}
\label{eqdefgenmeanfpool6b00}
|b^*(x,N)-b^*(x)|\le \ep (N),
\end{equation}
with some $\ep(N)\to 0$, as $N\to \infty$ and some function $b^*(x)$,
and let the functions $R_i(x,b)$, $b^*(x,N)$ and $b^*(x)$ belong to $C_{bLip}$ in all variables with
norms uniformly bounded by some $\om >0$. Suppose the initial data $x(N)$ of the Markov chains  $X^*_{N}(t,x(N))$
converge to a certain $x$ in $\R^d$, as $N\to \infty$. Then these
Markov chains  converge in distribution to the deterministic
evolution $X_t(x)$ solving the equation
\begin{equation}
\label{eqdefgenmeanfpool6brep}
\dot x_j=\sum_{i} \ka x_i x_j[R_j(x,b^*(x))-R_i(x,b^*(x))], \quad j=1,...,d,
\end{equation}
with initial condition $x$. This equation is globally well-posed: for any initial $x\in \Si_d$, the solution $X_t(x)$ exists and
belongs to $\Si_d$ for all times $t$.

For smooth or Lipschitz $g$, the following rates of convergence are valid:
\begin{equation}
\label{th1eq1}
 |\E g (X_N^*(t,x(N))-g(X_t(x(N)))| \le t C(\om,t) \left(\frac{d}{\sqrt N}+\ep(N)\right) \|g\|_{C^2(\Si_d)},
 \end{equation}
\begin{equation}
\label{th1eq1o}
 |\E g (X_N^*(t,x(N))-g(X_t(x(N)))| \le C(\om,t) \left(\frac{d t^{2/3}}{N^{1/3}}+t\ep(N)\right) \|g\|_{bLip},
 \end{equation}
 \begin{equation}
\label{th1eq1a}
|g(X_t(x(N))-g(X_t(x))|\le C(\om,t) \|g\|_{bLip}|x(N)-x|
\end{equation}
with constants $C(\om,t)$ uniformly bounded for bounded sets of $\om$ and $t$.
\end{theorem}

\begin{remark}
(i) We separate \eqref{th1eq1a} from \eqref{th1eq1} to stress that \eqref{th1eq1} holds without
the assumption of the convergence $x(N)\to x$. The dependence on $t$ and $d$ is not essential here,
but the latter becomes crucial for dealing with infinite state-spaces, while the former for dealing
with a forward looking principal.
(ii) The convergence result of (i) follows more-or-less directly from the general theory
(the settings of \cite{BeWe} or Section 11.9 of \cite{KolMal1} are just slightly different).
We give an analytic proof aiming at the effective rates of weak convergence, improving essentially the results
of \cite{Ko12} that dealt with smooth coefficients $R$.
\end{remark}

\begin{proof}
The well-posedness of  \eqref{eqdefgenmeanfpool6brep} is more or less obvious, and
it is a particular case of more general Theorem 6.1 of \cite{Ko10} or
Lemma \ref{lemonwellposedinl1} of Appendix (with the barrier $L$ being identically $1$).
Once the well-posedness is established, the  Lipshitz continuity \eqref{th1eq1a} of the solutions is
a standard fact from the theory of ODEs.

Next, since any function $g\in C(\R^d)$ can be approximated by functions from $C^2(\R^d)$,
the convergence of Markov chains from Statement (i) follows from \eqref{th1eq1} and \eqref{th1eq1a}.
Thus it remains to show \eqref{th1eq1} and \eqref{th1eq1o}.

Let us start with some calculations concerning $L_{b,N}$ assuming that
$\lim_{N\to\infty} b(x,N)=b(x)$ exists and that
$f\in C^1(\Si_d)$. Then we find, expanding $f$ in Taylor series, that
\[
\lim_{N \to \infty,\,  n/N \to x} L_{b,N}f (n/N)=\La_{b} f(x),
\]
where
\begin{equation}
\label{eqdefgenmeanfpool4}
\La_{b} f (x)=\sum_{i,j: R_j(x,b(x))>R_i(x,b(x))} \ka x_i x_j [R_j(x,b(x))-R_i(x,b(x))]
\left[\frac{\pa f}{\pa x_j}-\frac{\pa f}{\pa x_i}\right](x),
\end{equation}
or equivalently
\begin{equation}
\label{eqdefgenmeanfpool5}
\La_{b} f (x)=\sum_{i,j=1}^d \ka x_i x_j [R_j(x,b(x))-R_i(x,b(x))]\frac{\pa f}{\pa x_j}(x).
\end{equation}

Thus the limiting operator $\La_{b} f$ is the first-order PDO with characteristics solving the equations
\eqref{eqdefgenmeanfpool6}, which turn to the required equations \eqref{eqdefgenmeanfpool6brep} when $b=b^*$.
What is left is the rigorous proof that the convergence of the generators  $L_{b^*,N}$ to $\La_{b^*}$
on smooth functions $f$ implies the convergence of the corresponding semigroups.

The main idea is to approximate all Lipschitz continuous functions involved by the smooth ones.
Namely, choosing an arbitrary mollifier $\chi$  (non-negative infinitely smooth even function on $\R$ with
 a compact support and $\int \chi (w) \, dw=1$) and the corresponding mollifier $\phi(y)=\prod \chi (y_j)$
 on $R^{d-1}$, let us define, for any function $V$ on $\Si_d$,
 its approximation
 \[
 \Phi_{\de}[V](x)=\int_{R^{d-1}}\frac{1}{\de^{d-1}} \phi \left(\frac{y}{\de}\right)V(x-y) \, dy
 =  \int_{\R^{d-1}}\frac{1}{\de^{d-1}} \phi \left(\frac{x-y}{\de}\right)V(y) \, dy.
 \]
 Notice that $\Si_d$ is $(d-1)$-dimensional object, so that any $V$ on it can be considered
 as a function of first $(d-1)$ coordinates of a vector $x\in \Si_d$ (continued to $\R^{d-1}$
 in an arbitrary continuous way).
 It follows that
 \begin{equation}
\label{eq3thMarkDeconEvolBack1}
\|\Phi_{\de}[V]\|_{C^1} =|\Phi_{\de}[V]\|_{bLip}\le  \|V\|_{bLip}
\end{equation}
for any $\de$ and
\[
 |\Phi_{\de}[V](x)-V(x)|\le \int\frac{1}{\de^{d-1}} \phi \left(\frac{y}{\de}\right)|V(x-y)-V(x)| \, dy
 \]
 \begin{equation}
\label{eq2thMarkDeconEvolBack1}
 \le \|V\|_{Lip}  \int\frac{1}{\de^{d-1}} \phi \left(\frac{y}{\de}\right) |y|_1\, dy
 \le \de (d-1) \|V\|_{Lip}   \int |w| \chi (w) \, dw.
 \end{equation}

\begin{remark}
We care about dimension $d$ in the estimates only for future use (here it is irrelevant). By a different choice
of mollifier $\phi$ one can get rid of $d$ in \eqref{eq2thMarkDeconEvolBack1}, but then it would pop up
in \eqref{eq4thMarkDeconEvolBack1}, which is avoided with our $\phi$.
\end{remark}

Next, the norm
$ \|\Phi_{\de}[V]\|_{C^2}$
does not exceed the sum of the norm $ \|\Phi_{\de}[V]\|_{C^1}$
and the supremum of the Lipschitz constants of the functions
\[
\frac{\pa}{\pa x_j} \Phi_{\de}[V](x)=\int\frac{1}{\de^d} \left(\frac{\pa}{\pa x_j} \phi\right) \left(\frac{y}{\de}\right)V(x-y) \, dy.
\]
Hence
\begin{equation}
\label{eq4thMarkDeconEvolBack1}
\|\Phi_{\de}[V]\|_{C^2} \le \|V\|_{bLip}\left(1 + \frac{1}{\de} \int |\chi'(w) | \, dw\right).
\end{equation}

Let $U_N^t$ denote the semigroup of the chain $X^*_{N}(t,x)$: $U_N^tg(x)=\E g(X^*_{N}(t,x))$,
and $U^t$ the semigroup of the deterministic process generated by \eqref{eqdefgenmeanfpool6brep}:
$U^tg(x)=g(X_t(x))$. Let $U_{N,\de}^t$ and $U_{\de}^t$ be the same semigroups but built with respect to the
functions
\[
\Phi_{\de}[R_j](x)=\int\frac{1}{\de^d} \phi \left(\frac{y}{\de}\right)R_j(x-y,b^*(x-y)) \, dy
\]
rather than $R_j(x, b^*(x,N))$ and $R_j(x, b^*(x))$ respectively.
Similarly we denote by $L^{\de}_{b^*,N}$ and $\La^{\de}_{b^*}$ the corresponding generators
and by $X_t^{\de}(x)$ the solution of \eqref{eqdefgenmeanfpool6brep} with $\Phi_{\de}[R_j]$ used instead of $R_j$.

Then
\[
|\frac{d}{dt} (X_t(x)- X_t^{\de}(x))|_1\le 2\de + 4 \om |X_t(x)- X_t^{\de}(x)|_1,
\]
implying that $|X_t(x)- X_t^{\de}(x)|_1 \le \de t C(\om,t)$ and hence
\begin{equation}
\label{th1eq2}
|U^tg(x)-U^t_{\de}g(x)|=|g(X_t(x)-g(X_t^{\de}(x))|\le \|g\|_{bLip} \de t C(\om,t).
\end{equation}
Moreover, by Lemma \ref{lemonsenseBanach1} (its simplest finite dimensional version) and \eqref{eq4thMarkDeconEvolBack1}
\begin{equation}
\label{th1eq3}
|U^t_{\de}g(x)|_{C^2} \le C(\om,t) \left( \|g\|_{C^2} +\frac{1}{\de} \|g\|_{bLip} \right).
\end{equation}

Next we use \eqref{eqcompsem2} to get
\[
\|U_N^tg -U^t_{\de}g\|\le t \sup_{s\in [0,t]}\|(L_{b^*,N}-\La^{\de}_{b^*})U^s_{\de} g\|
\]
\begin{equation}
\label{eqcompsem2rep}
\le  t \sup_{s\in [0,t]}
\left( \|(L_{b^*,N}-L^{\de}_{b^*,N})U^s_{\de} g\| + \|(L^{\de}_{b^*,N} -\La^{\de}_{b^*})U^s_{\de} g\|\right).
\end{equation}

Then
\[
 \|(L_{b^*,N}-L^{\de}_{b^*,N})U^s_{\de} g\|\le C(\om) (\ep(N)+d \de) \|U^s_{\de} g\|_{bLip}
 \le C(\om,s)  (\ep(N)+d \de) \|g\|_{bLip},
 \]
 and (using \eqref{th1eq3})
 \[
\|(L^{\de}_{b^*,N} -\La^{\de}_{b^*})U^s_{\de} g\|  \le C(\om,t) \frac{1}{N} \|U^s_{\de} g\|_{C^2}
\le C(\om,t) \frac{1}{N} \|g\|_{C^2} (1+1/\de).
\]
Thus choosing $\de =1/\sqrt N$, makes the decay rate of $\de$ and $1/(N\de)$ equal yielding \eqref{th1eq1}.

Finally, if $g$ is only Lipschitz, we approximate it by $\Phi_{\tilde \de}[g]$, so that the second derivative of
$\Phi_{\tilde \de}[g]$ is bounded by $\|g\|_{bLip}/\tilde \de$. Thus the rates of convergence for $g$ become of order
\[
[d\tilde \de + t(\ep(N)+\de d +1/(N\de \tilde \de))]\|g\|_{bLip}.
\]
Choosing  $\de =(tN)^{-1/3}$, $\tilde \de =t^{2/3}N^{-1/3}$
 makes the decay rate of all terms (apart from $\ep(N)$) equal
  yielding \eqref{th1eq1o} and completing the proof.
\end{proof}

Theorem \ref{th1} suggests that eventually the evolution will settle down near some stable equilibrium points of
dynamic systems \eqref{eqdefgenmeanfpool6brep}. Analysis of stability of these equilibria will be carried out elsewhere.
As was mentioned, for a particular case of evolutionary inspection games it was worked out in \cite{KoPaYa}.
Let us observe only that system \eqref{eqdefgenmeanfpool6brep} is quite specific in the sense that its singular points
can be easily identified. In fact, for a subset $I\subset \{1, \cdots , d\}$, let
\[
\Om_I =\{ x \in \Si_d: x_k =0 \Longleftrightarrow  k\in I, \text{and} \, R_j(x,b^*(x))=R_i(x,b^*(x)) \, \text{for} \,  i,j\notin I \}.
\]

\begin{theorem}
\label{th11}
A vector $x$ with non-negative coordinates is a singular point of \eqref{eqdefgenmeanfpool6brep},
that is, it satisfies the system of equations
\begin{equation}
\label{eqdefgenmeanfpoolsing1}
\sum_{i} \ka x_i x_j[R_j(x,b^*(x))-R_i(x,b^*(x))]=0, \quad j=1,...,d,
\end{equation}
if and only if $x\in \Om_I$ for some $I\subset \{1, \cdots , d\}$.
\end{theorem}

\begin{proof}
Since for any $I$ such that $x_k=0$ for $k\in I$ the system  \eqref{eqdefgenmeanfpoolsing1}
reduces to the same system but with coordinates $k\notin I$, it is sufficient to show the result for
the empty $I$. In this situation, system  \eqref{eqdefgenmeanfpoolsing1} reduces to
\begin{equation}
\label{eqdefgenmeanfpoolsing2}
\sum_{i} x_i [R_j(x,b^*(x))-R_i(x,b^*(x))]=0, \quad j=1,...,d.
\end{equation}
Subtracting $j$th and $k$th equations of this system yields
\[
(x_1+\cdots + x_d)  [R_j(x,b^*(x))-R_k(x,b^*(x))]=0,
\]
and thus
\[
R_j(x,b^*(x))=R_k(x,b^*(x)),
\]
as required.
\end{proof}

So far we have deduced the dynamics arising from a certain Markov model of interaction.
As it is known, the internal
(not lying on the boundary of the simplex) singular points of the standard replicator dynamics of evolutionary game theory
correspond to the mixed-strategy Nash equilibria of the initial game with a fixed number of players
(in most examples just two-player game). Therefore, it is natural to ask whether a similar interpretation can be given to fixed points of
Theorem \ref{th11}. Because of the additional nonlinear mean-field dependence of $R$ on $x$
the interpretation of $x$ as mixed strategies is not at all clear.
However, consider explicitly the following game $\Ga_N$ of $N+1$ players (that was tacitly borne in mind
 when discussing dynamics).
 When the major player chooses the strategy $b$ and
each of $N$ small players chooses the state $i$, the major player receives the payoff
$B(x,b,N)$ and each player in the state $i$ receives $R_i(x,b)$, $i=1, \cdots, d$ (as above, with $x=n/N$ and
$n=(n_1, \cdots , n_d)$ the realized occupation numbers of all the states). Thus a strategy profile
of small players  in this game can be specified either by a sequence of $N$ numbers
(expressing the choice of the state by each agent), or more succinctly, by the resulting collection
of frequencies $x=n/N$.

As usual one defines a Nash equilibrium in $\Ga_N$ as a profile of strategies $(x_N,b_N)$ such that for any player
changing its choice unilaterally would not be beneficial, that is
\[
b_N =b_N^*(x_N)= argmax \, B(x_N,b,N)
\]
and for any $i,j\in \{1, \cdots , d\}$
\begin{equation}
\label{eqNashevolprin1}
R_j(x-e_i/N+e_j/N,b_N)\le R_i(x,b_N).
\end{equation}
A profile is an $\ep$-Nash if these inequalities hold up to an additive correction term not exceeding $\ep$.
It turns out that the singular points of \eqref{eqdefgenmeanfpool6brep} describe all approximate
Nash equilibria for $\Ga_N$ in the following precise sense:

\begin{theorem}
\label{th111}
Let $R(x,b)$ be Lipschitz continuous in $x$ uniformly $b$. Set
$\hat R=\sup_{i,b} \|R_i(.,b)\|_{Lip}$.
For $I\subset \{1, \cdots , d\}$, let
\[
\hat \Om_I =\{x\in \Om_I:  R_k(x,b^*(x)) \le R_i(x,b^*(x)) \, \text{for} \, k\in I, j\notin I\}.
\]
Then the following assertions hold.

(i) The limit points
of any sequence $x_N$ such that $(x_N, b^*(x_N))$ is a Nash equilibrium for $\Ga_N$ belong to $\hat \Om_I$ for some $I$.
In particular, if all $x_N$ are internal points of $\Si_d$, then any limiting point belongs to $\Om_{\emptyset}$.

(ii) For any $I$ and $x\in \Om_I$ there exists an $2\hat R d/N$-Nash equilibrium $(x_N,b_N^*(x_N))$ to $\Ga_N$
such  that the difference of any coordinates of $x_N$ and $x$ does not exceed $1/N$ in magnitude.
\end{theorem}

\begin{proof}
(i) Let us consider a sequence of Nash equilibria  $(x_N,b^*(x_N))$ such that the coordinates of all $x_N$ in $I$ vanish.
By \eqref{eqNashevolprin1} and the definition of $\hat R$,
\begin{equation}
\label{eqNashevolprin2}
|R_j(x_N,b_N^*(x_N)) - R_i(x_N,b_N^*(x_N))| \le \frac{2}{N}\hat R
\end{equation}
for any $i,j \notin I$ and
\begin{equation}
\label{eqNashevolprin2a}
R_k(x_N,b_N^*(x_N)) \le R_i(x_N,b_N^*(x_N)) + \frac{2}{N}\hat R, \quad k\in I, i\notin I.
\end{equation}
Hence  $x\in \hat \Om_I$ for any limiting point $(x,b)$.

(ii) If $x\in \hat \Om_I$ one can construct its $1/N$-rational approximation,
namely a sequence $x_N\in \Si_d \cap \Z_+^d/N$ such that the difference
of any coordinates of $x_N$ and $x$ does not exceed $1/N$ in magnitude.
For any such $x_N$, the profile $(x_N,b^*(x_N))$ is an $2\hat R d/N$-Nash equilibrium for $\Ga_N$.
\end{proof}

Theorem \ref{th111} provides a game-theoretic interpretation of the fixed points of dynamics
\eqref{eqdefgenmeanfpool6brep}, which is independent of any myopic hypothesis used to
justify this dynamics.

Of course, the set of 'almost equilibria' $\Om$ may be empty or contain many points.
Thus one can naturally pose here the analog of the question which is well discussed
in the literature on the standard evolutionary dynamics (see \cite{BinSam} and references therein),
namely which equilibria can be chosen  in the long run (the analogs of stochastically
stable equilibria in the sense of \cite{FoYou}) if
small mutations are included in the evolution of the Markov approximation.

\subsection{Basic examples}
\label{secexampmatpressureevol}

In the standard setting of inspection games with a possibly tax-evading inspectee
(analyzed in detail in \cite{KoPaYa} under some particular assumptions),
the payoff $R$ looks as follows:
\begin{equation}
\label{eqdefinsppay}
R_j(x,b) =r+(1-p_j(x,b))r_j-p_j(x,b) f(r_j),
\end{equation}
where $r$ is the legal payoff of an inspectee, various $r_j$
denote various amounts of not declared profit, $j=1, \cdots , d$,
$p_j(x,b)$ is the probability for the illegal behavior of an inspectee
to be found when the inspector uses budget $b$ for searching operation
and $f(r_j)$ is the fine that the guilty inspectee has to pay when being discovered.

In the standard model of corruption 'with benevolent principal', see e. g. \cite{Aidt},
one sets the payoff of a possibly corrupted inspector (now taking the role
of a small player) as
\[
(1-p)(r+w)+p(w_0-f),
\]
where $r$ is now the bribe an inspector asks from a firm to agree not to publicize its profit
(and thus allowing her not to pay tax), $w$ is the wage of an inspector, $f$ the fine she has to pay
when the corruption is discovered and $p$ the probability of a corrupted behavior to be discovered by the
benevolent principal (say, governmental official). Finally it is assumed that when the corrupted behavior
is discovered the agent not only pays fine, but is also fired from the job
and has to accept a lower level activity with the reservation wage $w_0$. In our strategic model
we make $r$ to be the strategy of an inspector with possible levels $r_1, \cdots, r_d$
(the amount of bribes she is taking) and the probability $p$ of discovery to be dependent
on the effort (say, budget $b$) of the principal and the overall level of corruption $x$,
with fine too depending on the level of illegal behavior. This natural extension of the standard model
 leads to the payoff
\begin{equation}
\label{eqdefcorrpay}
R_j(x,b) =(1-p_j(x,b))(r_j+w)+p_j(x,b)(w_0- f(r_j)),
\end{equation}
 which is essentially identical to \eqref{eqdefinsppay}.

In the more general pressure and resistance games, the payoff
$R_j(x,b)$ has the following special features: $R$ increases
in $j$ and decreases in $b$. The dependence of $R$ and $b^*$ on $x$
is more subtle, as it may take into account social norms of various character.
In case of the pressure game with resistance and collaboration, the strategic
parameter $r$ of small players naturally decomposes into two coordinates $r=(r^1,r^2)$,
the first one reflecting the level of resistance and the second
the level of collaboration.
If the correlation between these activities are not taken into account
the payoff $R$ can be decomposed into the sum
of rewards $R=R^1_j(x,b)+R^2_j(x,b)$ with $R^1$
having the same features as $R$ above, but with
$R^2$ increasing both in $j$ and $b$.

As another set of examples let us look at the applications to the botnet defense
(for example, against the famous conflicker botnet), widely discussed in the contemporary
 literature, since botnets (zombie networks) are considered to pose the biggest
 threat to the international cyber-security, see e. g. review of the abundant
 bibliography in \cite{BenKaHo}. The comprehensive game theoretical framework of
 \cite{BenKaHo} (that extends several previous simplified models) models the group
of users subject to cybercriminal attack of botnet herders
as a differential game of two players, the group of cybercriminals and the group of defenders.
Our approach adds to this analysis the networking aspects by allowing the defenders to communicate
and eventually copy more beneficial strategies.
More concretely, our general model of inspection or corruption
 becomes almost directly applicable in this setting by the clever linguistic change of 'detected' to
'infected' and by considering the cybecriminal as the 'principal agent'!
 Namely, let $r_j$ (the index $j$ being taken from some discrete set here, though
 more advanced theory of the next sections allows for a continuous parameter $j$)
 denote the level of defense applied by an individual (computer owner)
against botnet herders (an analog of the parameter $\ga$ of  \cite{BenKaHo}), which can be
the level of antivirus programs installed or the measures envisaged to quickly report and repair
a problem once detected (or possibly a multidimensional parameter reflecting several defense measures).
Similarly to our previous models, let $p_j(x,b)$ denote the probability for a computer of being infected
given the level of defense measures $r_j$, the effort level $b$ of the herder (say, budget or time spent)
and the overall distribution $x$ of infected machines (this 'mean-field' parameter is crucial in the present setting,
since infection propagates as a kind of epidemic). Then, for a player with a strategy $j$,
the cost of being (inevitably) involved in the conflict can be naturally estimated by the formula
\begin{equation}
\label{eqdefcybersecpay}
R_j(x,b) =p_j(x,b)c+r_j,
\end{equation}
where $c$ is the cost (inevitable losses) of being infected (thus one should aim
at minimizing this $R_j$, rather then maximizing it, as in our previous models). Of course, one can extend the model
to various classes of customers (or various classes of computers) for which values of $c$ or $r_j$ may vary
and by taking into account more concrete mechanisms of virus spreading, as described e. g. in
\cite {LiLiSt} and \cite{LyWi}.

%Similar models can be applied to the analysis of defense against a biological weapon, for instance by adding
%the active agent (principal interested in spreading the disease), into the general mean-field epidemic model
%of \cite{LiuTaYa} that extends the well established SIS (susceptible-infectious-susceptible) and
%SIR (susceptible-infectious-recovered) models.

Yet another set of examples represent the models of terrorists' attacks and counterterrorism measures,
see e. g. \cite{ArSa05},  \cite{SaAr03}, \cite{SandLa88}, \cite{BrKi88}
for the general background on game -theoretic models of terrorism,
and \cite{FaAr12} for more recent developments. We again suggest here a natural extension to basic models to
the possibility of interacting large number of players and of various levels of attacks, the latter extension being in the line with
argument from \cite{ClYoGl07} advocating consideration of 'spectacular attacks' as part of a continuous scale
of attacks of various levels. In the literature, the counterterrorists' measures are usually
 decomposed into two groups, so called proactive (or preemptive),
like direct retaliation against the state-sponsor
and defensive (also referred to as deterrence), like strengthening security at an airport, with
the choice between the two considered as the main strategic parameter.
As stressed in \cite{RosSand04} the first group of action is 'characterized in
the literature as a pure public good, because a weakened terrorist group poses less of a
threat to all potential targets', but on the other hand, it 'may have a downside by creating more grievances in reaction
to heavy-handed tactics or unintended collateral damage' (because it means to
'bomb alleged terrorist assets, hold suspects without
charging them, assassinate suspected terrorists, curb civil freedoms, or impose retribution
on alleged sponsors'), which may result in the increase of terrorists' recruitment.
Thus, the model of \cite{RosSand04} includes the recruitment benefits of terrorists as a positively correlated function
of preemption efforts. A direct extension of the model of \cite{RosSand04} in the line indicated above
(large number of players and the levels of attacks) suggests to write down the reward of a terrorist, or a terrorist group,
considered as a representative of a large number of small players, using one of the
levels of attack $j=1, \cdots, d$ (in \cite{RosSand04} there are two levels, normal and spectacular only), to be
\begin{equation}
\label{eqdefterrpay}
R_j(x,b) =(1-p_j(x,b))r_j^{fail}(b)+p_j(x,b)(S_j+r_j^{succ}(b)),
\end{equation}
where $p_j(x,b)$ is the probability of a successful attack (which depends on the level $b$ of preemptive efforts
of the principal $b$ and the total
 distribution of terrorists playing different strategies), $S_j$ is the direct benefits in case of a success and
 $r_j^{fail}(b)$, $r_j^{succ}(b)$ are the recruitment benefits in the cases of failure or success respectively.
 The costs of principal are given by
 \[
 B(x,b)=\sum_j x_j \left[(1- p_j(x,b))b+p_j(b)(b+S_j)\right].
 \]
 It is seen directly that we are again in the same situation as described by \eqref{eqdefcorrpay}
 (up to constants and notations). The model extends naturally to account for possibility of the actions of two types,
 preemption and deterrence. Of importance should be its extension to several major players
  (for instance, USA and EU are considered in \cite{ArSa05}).

As was mentioned in introduction, there exists a large class of problems, where the state space of small players become
infinite. We shall pay most of our attention to the major particular case (possibly the mostly relevant
one for practical purposes) of a countable state space arising in the analysis of the models of evolutionary growth.
For this class of models the number $N$ of agents become variable (and usually growing in the result of the evolution)
and the major characteristics of the system becomes just the distribution $x=(x_1, x_2, \cdots )$ of the sizes
of the groups. The analysis of the evolution of these models is well -developed and has a long history,
see \cite{SimRoy}. Mathematically the analysis is similar to finite state spaces, though serious
 technical complications may arise. We develop the 'strategically enhanced model' in
Section \ref{secmodgrowthpres} analyzing such evolutions under the 'pressure' of strategically varying parameters
set by the principal.

\subsection{Compact state-space}
\label{secbestrespprin1ad}

Let us extend the analysis given above to the case of continuous state space of small
players, assuming it to be a compact convex subset $Z$, of a Euclidean space $\R^n$.
Let $\PC(Z)$ denote the set of probability laws on $Z$ equipped with its weak topology.
For each $N$ the state space of $N$ agents becomes $Z^N$. However, assuming agents to be indistinguishable,
the state space is better described as the set of equivalence classes of $Z^N$ with respect to all permutations
that can be naturally identified with the set $M_N$ of the normalized sums of $N$ Dirac measures
\[
\frac{1}{N}(\de_{x_1}+\cdots +\de_{x_N}).
\]
For $\x=(x_1, \cdots, x_N)$ let us use shorter notation $\de_{\x}$ for the sum $\de_{x_1}+\cdots +\de_{x_N}$.
Assume that continuous functions $R(x,\mu,b)$ on
$(Z\times \MC^+(Z) \times \R^r)$ and $B(\mu,b,N)$ on $(\MC^+(Z)\times \R^r \times \N)$ are given such that
$R(x_j, (\de_{x_1}+\cdots +\de_{x_N})/N, b)$
is the payoff for $x_j$ in the group $\x=(x_1, \cdots, x_N)$ given the level of efforts $b\in \R^r$ of the major player,
and $B((\de_{x_1}+\cdots +\de_{x_N})/N, b,N)$ is the payoff of the major player applying the effort level $b$ to the
the group $\x=(x_1, \cdots, x_N)$.
Assume again that $B$ is a smooth and strictly concave function of $b$, so that
\begin{equation}
\label{eqdefgenmeanfpool6b0cont}
b^*(\de_{\x}/N,N)=argmax_b \, B(\de_{\x}/N,b,N)
\end{equation}
is well defined and that the limit
\begin{equation}
\label{eqdefgenmeanfpool6b00cont}
\lim_{N\to\infty} b^*(\mu,N)=b^*(\mu)
\end{equation}
exists uniformly in $\mu \in \PC(Z)$.

The direct analog of the generator \eqref{eqdefgenmeanfpool2} with $b=b^*$
(describing the Markov chain produced by pairwise exchange of information)
to the continuous state-space is clearly
the operator

\[
L_{b^*,N}f (\de_{\x}/N)=\frac{\ka}{N} \sum_{(i,j)}[f(\de_{\x}/N-\de_{x_i}/N+\de_{x_j}/N)-f(\de_{\x}/N)]
\]
\begin{equation}
\label{eqdefgenmeanfpool2cont}
\times [R(x_j,\de_{\x}/N,b^*(\de_{\x}/N,N))-R(x_i,\de_{\x}/N,b^*(\de_{\x}/N,N))],
\end{equation}
where $\x=(x_1, \cdots, x_N)$ and the sum is over all pairs $(i,j)$ of indices ordered in such a way that
\[
R(x_j,\de_{\x}/N,b^*(\de_{\x}/N,N))>R(x_i,\de_{\x}/N,b^*(\de_{\x}/N,N))
\]
(the order is irrelevant if the corresponding values of $R$ coincide).

Let us denote by $X^*_{N}(t,\de_{\x}/N)$ the Markov chain on $M_N$ generated by \eqref{eqdefgenmeanfpool2cont}.

In order to see what happens with  generator \eqref{eqdefgenmeanfpool2cont} in the limit $N\to \infty$,
take a linear function $f$ on measures given by the integration, that is,
\begin{equation}
\label{eqlinfunmes}
f(\mu)=F_g(\mu)=\int g(x) \mu (dx).
\end{equation}
Then
\[
L_{b^*,N}F_g (\de_{\x}/N)=\frac{\ka}{N^2} \sum_{(i,j)}
\]
\[
\times
[R(x_j,\de_{\x}/N,b^*(\de_{\x}/N,N))-R(x_i,\de_{\x}/N,b^*(\de_{\x}/N,N))]
[g(x_j)-g(x_i)].
\]
Since the product of the square brackets is invariant under the change of the order of $(i,j)$,
this rewrites in a simpler form as
\begin{equation}
\label{eqdefgenmeanfpool2cont1}
L_{b^*,N}F_g (\de_{\x}/N)=\frac{\ka}{2N^2} \sum_{i,j=1}^N
[R(x_j,\de_{\x}/N,b^*(\de_{\x}/N,N))-R(x_i,\de_{\x}/N,b^*(\de_{\x}/N,N))](g(x_j)-g(x_i)],
\end{equation}
and consequently as
\[
 L_{b^*,N}F_g (\de_{\x}/N)
=\frac{\ka}{2}  \int \int [g(z_2)-g(z_1)]
\]
\[
\times [R(z_2,\de_{\x}/N,b^*(\de_{\x}/N,N))-R(z_1,\de_{\x}/N,b^*(\de_{\x}/N,N))]
\frac{1}{N}\de_{\x}(dz_1)\frac{1}{N}\de_{\x}(dz_2).
\]
Thus if $\de_{\x}/N \to \mu$ as $N\to \infty$ with any $\mu\in M(Z)$ this turns to
\begin{equation}
\label{eqdefgenmeanfpool2cont3}
 L_{b^*}F_g (\mu)
=\frac{\ka}{2}  \int_Z \int_Z [g(z_2)-g(z_1)][R(z_2,\mu,b^*(\mu))-R(z_1,\mu,b^*(\mu))]
\mu(dz_1)\mu(dz_2),
\end{equation}
or equivalently
\begin{equation}
\label{eqdefgenmeanfpool2cont4}
 L_{b^*}F_g (\mu)
=\ka \int_Z \int_Z g(z_2)[R(z_2,\mu,b^*(\mu))-R(z_1,\mu,b^*(\mu))]\mu(dz_1)\mu(dz_2).
\end{equation}

These calculations make the following result plausible. Unlike finite-state-space case, we
give two different convergence rates depending basically on whether weak or strong regularity
is assumed on the coefficients. We use the notations for the spaces of functions on measures introduced
in Appendices \ref{secnotspace} and \ref{secvarder}.
Assume for definiteness that $Z$ belongs to the cube $[0,K]^n$ of $\R^n$.

\begin{theorem}
\label{th2}
(i) Suppose the functions $R(x,\mu,b)$ and $b^*(\mu)$ are bounded weakly Lipschitz with respect to all their variables
with the bounds and Lipschitz constants bounded by some $\om$.
 Suppose the initial data $\de_{\x(N)}/N$ of the Markov chains  $X^*_{N}(t,\de_{\x(N)}/N)$
converge weakly to a certain $\mu\in \PC(Z)$, as $N\to \infty$. Then these
Markov chains  converge in distribution to the deterministic
evolution on $\PC(Z)$ solving the kinetic equation
\begin{equation}
\label{eqkineqga1}
\dot \mu_t (dz)=\ka \int_{y\in Z} [R(z,\mu_t,b^*(\mu_t))-R(y,\mu_t,b^*(\mu_t))]\mu_t(dy)\mu_t(dz),
\end{equation}
or equivalently in the weak form
\begin{equation}
\label{eqkineqga2}
\frac{d}{dt} \int g(z)\mu_t (dz)=\ka \int_{Z^2} g(z) [R(z,\mu_t,b^*(\mu_t))-R(y,\mu_t,b^*(\mu_t))]\mu_t(dy)\mu_t(dz).
\end{equation}
This equation is globally well-posed: for any initial $\mu\in \MC^+(Z)$
(in particular $\mu\in \PC(Z)$), the solution $\mu_t(\mu)$ exists and
belongs to $\MC^+(Z)$ (respectively, $\PC(Z)$) for all times $t$.

Moreover, if $g\in C^2_{weak}(\MC_1^+(Z))\cap C^{bLip}_{weak}(\MC_1^+(Z))$, the following rate of convergence is valid:
\[
 |\E g (X^*_{N}(t,\de_{\x(N)}/N))-g(\mu_t(\de_{\x(N)}/N))|
 \]
 \begin{equation}
\label{eqkineqga3}
 \le t C(\om,t) \left(\frac{1}{N^{1/(2+n)}}+\ep(N)\right) (\|g\|_{C^2_{weak}}+\|g\|_{weakLip}).
 \end{equation}
 If $g\in C^{bLip}_{weak}(\MC_1^+(Z))$, then
 \begin{equation}
\label{eqkineqga3a}
 |\E g (X^*_{N}(t,\de_{\x(N)}/N))-g(\mu_t(\de_{\x(N)}/N))|
 \le  C(\om,t) \left(\frac{t}{(tN)^{1/(2n+3)}}+t\ep(N)\right)\|g\|_{weakLip},
 \end{equation}
 \begin{equation}
\label{eqkineqga4}
|g(\mu_t(\de_{\x(N)}/N))-g(\mu_t(\mu))|\le C(\om,t) \|g\|_{bLip}|d_{bLip*}(\de_{\x(N)}/N,\mu),
\end{equation}
with constants $C(\om,t)$ uniformly bounded for bounded $\om$ and $t$.

(ii) Not assuming weak Lipschitz continuity, but assuming instead that $R$ and $b$ are strongly twice continuously differentiable,
one has the following rate of convergence
\begin{equation}
\label{eqkineqga5}
|\E g(X_N^*(t,\de_{\x(N)}/N)-g(\mu_t (\de_{\x(N)}/N))|\le tC(\om,t)\frac{1}{N} \|g\|_{C^2(\MC_1(Z))}.
\end{equation}
\end{theorem}

\begin{remark}
(i) A probabilistic proof of convergence is again well known (via the tightness of
the related martingale problems), see e.g. similar argument in Theorem 4.1. of  \cite{Ko06}, but it does not supply
the rates that are crucial for applications to optimal control. (ii) All estimates reduce to the estimates
of Theorem \ref{th2} by setting $n=0$, as expected (the dimension of a finite set is zero).
\end{remark}

\begin{proof}
 Well-posedness of \eqref{eqkineqga1} is a consequence of Lemma
\ref{lemonwellposedinl1} (with the barrier $L$ being identically $1$). Then
estimate \eqref{eqkineqga4} is the standard Lipschitz continuity of the solutions of ODE with Lipschitz coefficients
with respect to initial data.
Estimate \eqref{eqkineqga5} is obtained analogously to \eqref{eqkineqga3} using strong derivatives instead of weak ones,
but much simpler indeed, as the assumption of smoothness allows one to avoid any additional approximations.

Let us concentrate on \eqref{eqkineqga3} and \eqref{eqkineqga3a}.

The generator above is calculated only for linear functionals on measures.
To calculate it for arbitrary smooth functionals, one has to use the technique of variational derivatives
(recalled in Appendix).
Namely, for a smooth  $f$ the value of $L_{b^*,N}f(\mu)$ is given by Lemma \ref{lembasas}, that is, it coincides with
\[
L_{b^*}^{lim}f(\mu) =\ka \int_{Z^2} \frac{\de f(\mu)}{\de \mu (z_2)}[R(z_2,\mu,b^*(\mu))-R(z_1,\mu,b^*(\mu))]\mu(dz_1)\mu(dz_2)
\]
up to an additive correction of order $1/N$ depending on the second derivatives of $f$.

To deduce the convergence of processes from the convergence of generators on $f\in C^2_{weak}(\MC(Z))$
we follow the same strategy of approximation as above for Theorem \ref{th1}.
An additional ingredient is the approximation of a weakly Lipschitz function $F$ on $\MC(Z)$,
with the weak Lipschitz constant $\|F\|_{weakLip}$, by
finite-dimensional functionals (see Appendix I in \cite{Ko10}). Namely,
for $j\in \N$, let $\x^k_j=(K/j)k$, $k=(k_1, \cdots , k_n)$ with $k_l\in \{0, \cdots ,j\}$,
be the lattice of $(j+1)^n$ points in $[0,K]^n$ and $\phi_j^k$ be the collection
of $(j+1)^n$ functions on $\R^n$ given by
\[
\phi_k^j(x)=\prod_{i=1}^n \chi \left(\frac{j}{K}(x_i-k_i\frac{K}{N})\right),
\quad
\chi (z)=\left\{
\begin{aligned}
& 1-|z|, \, |z|\le 1 , \\
& 0, \quad |z| \ge 1.
\end{aligned}
\right.
\]

This choice of functions $\phi_k^j$ is not at all unique. It is just a concrete example of
non-negative functions satisfying the following conditions: for any $j$,
 $\sum_{k=(k_1, \cdots , k_n)} \phi_k^j=1$ and an arbitrary $x$ can belong to the supports of not more than $2^n$
of functions $\phi_k^j$; and
\begin{equation}
\label{eqpartunitspec}
|\phi_k^j(x)-\phi_k^j(y)|\le \frac{j}{K} |x-y|_1.
\end{equation}

Then one defines the finite-dimensional projections in the spaces of functions and measures on $Z\subset [0,K]^n$:
\[
P_j(f)=\sum_k f(x_j^k) \phi_j^k,
\quad P_j^*(\mu)=\sum_l (\phi_j^l, \mu)\de_{x_j^l},
\]
and the corresponding finite-dimensional projections on $C_{weak}(\MC^+_1(Z))$
\[
F(\mu) \mapsto F_j(\mu)=F(P_j^*(\mu)).
\]

The projections $P_j$ have the following properties:
\begin{equation}
\label{eqpartunitspec1}
\|P_j\| \le \|f \|, \quad \|P_jf-f\|\le 2^n n\frac{K}{j} \|f\|_{Lip}, \quad
\quad \|P_jf\|_{Lip}\le 2^{n+1} n \|f\|_{Lip}.
\end{equation}
The first one is obvious. The second one follows from the estimate
\[
\|P_jf-f\|= \sum_k |(f(\x^k_j)-f(x))\phi^k_j(x)|\le 2^d \max |(f(\x^k_j)-f(x)|,
\]
where $\max$ is over those $k$ that $x$ belongs to the support of $\phi^k_j$.
To prove the third inequality of \eqref{eqpartunitspec1}, take arbitrary $x,y$ with $|x-y|_1 \le nK/j$.
Then
\[
|P_jf(x)-P_jf(y)| =\sum_k [f(x_k^j)\phi_k^j(x)-f(x_k^j)\phi_k^j(y)].
\]
Here the sum is over not more than $2^{n+1}$ lattice points (maximum $2^n$ for either $x$ or $y$).
 Let $k_0$ be one of these points. Then
\[
|P_jf(x)-P_jf(y)|=|\sum_{k\neq k_0}[f(x_k^j)\phi_k^j(x)-f(x_k^j)\phi_k^j(y)]
+f(x_{k_0}^j)(\sum_{k\neq k_0}\phi_k^j(y))-\sum_{k\neq k_0}\phi_k^j(y))|
\]
\[
=|\sum_{k\neq k_0}(f(x_k^j)-f(x_{k_0}^j))(\phi_k^j(x)-\phi_k^j(y))|\le 2^{n+1}\|f\|_{Lip}\frac{K}{j} n \frac{j}{K} |x-y|_1,
\]
yielding the third estimate of \eqref{eqpartunitspec1}.
From \eqref{eqpartunitspec1} it follows that
\begin{equation}
\label{eqpartunitspec2}
d_{bLip*}(P^*_j\mu_1, P^*_j\mu_2) \le 2^{n+1}n\, d_{bLip*}(\mu_1, \mu_2),
\quad \|F_j\|_{weakLip} \le 2^{n+1}n \|F\|_{weakLip},
\end{equation}
\begin{equation}
\label{eqpartunitspec3}
d_{bLip*}(P^*_j\mu, \mu) \le 2^n n\, \frac{K}{j},
\quad \|F_j(\mu)-F(\mu)\| \le  2^n n\, \frac{K}{j} \|F\|_{weakLip}.
\end{equation}

 Now $F_j(\mu)$ can be written as some function $ F_j(\mu)= f_j(\{(\phi_k^j, \mu)\})$
 of $(j+1)^n$ variables $\u^j=\{u_k^j=(\phi_k^j, \mu)\}$ such that
 \[
 |f_j(\u^{j1})-f_j(\u^{j2})|
 \le 2^{n+1}n \|F\|_{weakLip} \|\sum (u^{j1}_k-u^{j2}_k) \de_{\x_k^j}\|_{bLip*}
\le 2^{n+1}n \|F\|_{weakLip} \|\u^{j1}-\u^{j2}\|_1.
\]
Thus $f$ is Lipschitz in $\u$ and we can apply the same smooth approximation as in the proof of
Theorem \ref{th1} above. Here the dimension becomes essential.
Namely, using literally the same argument as
in Theorem \ref{th1} we obtain
\[
 |\E g (X^*_{N}(t,\de_{\x(N)}/N))-g(\mu_t(\de_{\x(N)}/N))|
\]
\begin{equation}
\label{ekineqga3proof}
 \le t C(\om,t) \left(\frac{1}{j}+\ep(N)+\de (j+1)^n+\frac{1}{\de N}\right) (\|g\|_{C^2_{weak}}+\|g\|_{weakLip}).
 \end{equation}
 Choosing $j=N^{\be}$ and $\de=N^{-(1-\be)}$ with $\be=1/(2+n)$ makes the rates of decay of $1/j$, $\de j^n$ and $1/(N\de)$ equal
 yielding \eqref{eqkineqga3}.

 Finally, if $g$ is assumed to be only weakly Lipschitz, we approximate it by the smooth one, as above.
  Thus the rates of convergence for $g$ become of order
\[
[j^n\tilde \de + t(\ep(N)+1/j + \de j^n +1/(N\de \tilde \de))]\|g\|_{bLip}.
\]
Choosing
\[
j=(tN)^{1/(2n+3)}, \quad \de =j^{-(n+1)}, \quad \tilde \de =t\de =t j^{-(n+1)}
\]
makes the decay rate of all terms (apart from $\ep(N)$) equal
yielding \eqref{eqkineqga3a} and completing the proof.
\end{proof}

The extension of Theorem \ref{th11} to the present case is as follows.
\begin{theorem}
\label{th21}
A (non-negative) measure $\mu$ is a singular point of \eqref{eqkineqga1}, that is, it satisfies
\begin{equation}
\label{eq1th21}
\int_{y\in Z} [R(z,\mu,b^*(\mu))-R(y,\mu,b^*(\mu))]\mu(dy)\mu(dz)=0,
\end{equation}
if and only if the function $R(.,\mu,b^*(\mu))$ is constant on the support of $\mu$.
\end{theorem}

\begin{proof}
Denoting
\[
\|\mu\|=\int_Z \mu (dy), \quad (R, \mu) =\int_Z R(y,\mu,b^*(\mu)) \mu (dy),
\]
equation \eqref{eq1th21} rewrites as
\begin{equation}
\label{eq2th21}
R(z,\mu,b^*(\mu))\mu(dz)=\frac{(R,\mu)}{\|\mu\|}\mu (dz),
\end{equation}
and the result follows.
\end{proof}

The corresponding extension of Theorem \ref{th111} is now also straightforward.

\subsection{Optimal allocation and group interaction }
\label{secoptimalocandgroup}

So far our small players were indistinguishable. However, in many cases the small players can belong to different types.
These can be inspectees with various income brackets, the levels of danger or overflow of particular traffic path,
or the classes of computers susceptible to infection.
In this situation the problem for the principal becomes a policy problem, that is,
how to allocate efficiently her limited resources. Our theory extends to a setting with various types more-or-less straightforwardly.
We shall touch it briefly.

Let our players, apart from being distinguished by states $i\in \{1,\cdots, d\}$,
can be also classified by their types or classes $\al \in \{1, \cdots, \AC \}$.
The state space of the group becomes $\Z^d_+\times \Z^{\AC}_+$, the set of matrices $n=(n_{i\al})$, where $n_{i\al}$ is the number
 of players of type $\al$ in the state $i$ (for simplicity of notation we identify the state spaces of each type, which is not at all necessary).
 One can imagine several scenarios of communications between classes, two extreme cases being as follows:

 (C1) No-communication: the players of different classes can neither communicate nor observe the distribution of
  states in other classes, so that the interaction between types arises exclusively through the principal;

 (C2) Full communication: the players can change both their types and states via pairwise exchange of information, and
 can observe the total distribution of types and states.

 There are lots of intermediate cases, say, when types form a graph (or a network) with edges specifying the possible
  channels of information. Let us deal here only with cases (C1) and (C2). Starting with (C1), let
$N_{\al}$ denote the number of players in class $\al$ and $n_{\al}$ the vector $\{n_{i\al}\}, i=1, \cdots ,d$.
Let $x_{\al}=n_{\al}/N_{\al}$,
\[
x=(x_{i\al})= (n_{i\al}/N_{\al})\in (\Si_d) ^{\AC},
\]
and $b=(b_1, \cdots, b_{\AC})$ be the vector of the allocation of resources of the principal,
which may depend on $x$. Assuming that the principal uses the optimal policy
\begin{equation}
\label{eqdefgenmeanfcl0}
b^*(x)=argmax \, B(x,b)
\end{equation}
 arising from some concave (in the second variable)
payoff function $B$ on $(\Si_d) ^{\AC}\times \R^{\AC}$, the generator \eqref{eqdefgenmeanfpool2}
extends to

\[
L_{b*,N}f (x)=\sum_{\al=1}^{\AC}N_{\al}\ka_{\al} \sum_{i,j: R_j^{\al}(x_{\al},b^*(x))>R_i^{\al}(x_{\al},b^*(x))} x_{i\al} x_{j\al}
\]
\begin{equation}
\label{eqdefgenmeanfcl1}
\times
[R_j^{\al}(x_{\al},b^*(x))-R_i^{\al}(x_{\al},b^*(x))][f(x-e_i^{\al}/N_{\al}+e_j^{\al}/N_{\al})-f(x)],
\end{equation}
where $e_i^{\al}$ is now the standard basis in $\R^d\times \R^{\AC}$.
Passing to the limit as $N\to \infty$ under the assumption that
\[
\lim_{N\to \infty} N_{\al}/N=\om_{\al}
\]
with some constants $\om_{\al}$ we obtain a generalization of
\eqref{eqdefgenmeanfpool6brep} in the form
\begin{equation}
\label{eqdefgenmeanfcl2}
\dot x_{j\al}=\ka_{\al}\om_{\al} \sum_i x_{i\al} x_{j\al}[R^{\al}_j(x_{\al},b^*(x))-R^{\al}_i(x_{\al},b^*(x))],
\end{equation}
for $j=1,...,d$ and $\al=1, \cdots , \AC$,
coupled with \eqref{eqdefgenmeanfcl0}.

In case (C2), $x=(x_{i\al})\in \Si_{d\al}$, the generator becomes
\[
L_{b*,N}f (x)=\sum_{\al, \be=1}^{\AC}N\ka \sum_{i,j: R_j^{\al}(x,b^*(x))>R_i^{\be}(x,b^*(x))} x_{i\al} x_{j\al}
\]
\begin{equation}
\label{eqdefgenmeanfcl1a}
\times
[R_j^{\al}(x,b^*(x))-R_i^{\be}(x,b^*(x))][f(x-e_i^{\be}/N_{\al}+e_j^{\al}/N_{\al})-f(x)],
\end{equation}
and the limiting system of differential equations
\begin{equation}
\label{eqdefgenmeanfcl3}
\dot x_{j\al}=\ka \sum_{i, \be} x_{i\be} x_{j\al}[R^{\al}_j(x,b^*(x))-R^{\be}_i(x,b^*(x))].
\end{equation}

So far we have assumed that the propagation of strategies is due to pairwise interaction (say, exchange of opinions).
Let us now extend the model by allowing simultaneous interactions in groups of arbitrary size, with appropriate scaling
that makes the contribution of simultaneous group interaction comparable with the contribution of pairwise exchange.
For humans this $k$th order interaction seems to be even more realistic than in chemistry, where similar considerations
leads to the so-called mass-action law for the rates of chemical reactions, see \cite{GoShah11} for the latter.
Equations \eqref{eqdefgenmeanfpoolkthor4} below can be considered as a performance of the 'mass action law for agents'
playing against the principal.

Assume that any collection of $k$ small players $\{i_1, \cdots i_k\}$, with $k$ not exceeding certain level $K$, can
be formed randomly with uniform distribution (any collection of $k$ players is equally likely) and exchange
opinions with the effect that all members of the group will accept the strategy $j$
of the member with the highest payoff, so that $R_j(x,b)=\max_l R_{i_l}(x,b)$, with some rates
$\Pi_I=\Pi (R_{i_1}, \cdots , R_{i_k})$,
which are symmetric functions of their arguments that vanish whenever all $R_{i_l}(x,b)$ are equal.
If there are several members of the group with the same payoff, the choice can be fixed arbitrary, say by
choosing the member with the highest index $i$. For simplicity (to shorten the formulas below) 
let us assume that only the players from different states  
can interact. Therefore, instead of a Markov chain with generator \eqref{eqdefgenmeanfpool2}, 
we obtain the chain with the generator
\begin{equation}
\label{eqdefgenmeanfpoolkthor1}
L_{b,N}f (n)=N\ka \sum_{k=2}^K \sum_{I=\{i_1, \cdots , i_k\}} \prod_{l=1}^k x_{i_l} \Pi_{I}[f(x+ke_{j(I)}/N-\sum_{i\in I}e_i/N)-f(x)],
\end{equation}
where $I$ are now all possible subsets of $\{1, \cdots, d\}$ of size $k$. 

Assuming again that
$\lim_{N\to\infty} b(x,N)=b(x)$ exists and that $f\in C^1(\Si_d)$, we find now,
analogously to the calculations with \eqref{eqdefgenmeanfpool2} (that is by expanding $f$ in Taylor series), that
\[
\lim_{N \to \infty,\,  n/N \to x} L_{b,N}f (n/N)=\La_{b} f(x),
\]
where
\begin{equation}
\label{eqdefgenmeanfpoolkthor2}
\La_b f (x)=\ka \sum_{k=2}^K \sum_{I=\{i_1, \cdots , i_k\}}
 \Pi_{I}\left[k\frac{\pa f}{\pa x_{j(I)}}-\sum_{i\in I}\frac{\pa f}{\pa x_i}\right](x)\prod_{l=1}^k x_{i_l},
\end{equation}
or equivalently
\begin{equation}
\label{eqdefgenmeanfpoolkthor3}
\La_{b} f (x) = \sum_{k=2}^K \ka \sum_{m=1}^d  \frac{\pa f}{\pa x_m}(x) x_m
\left[
k \sum'_{I=\{i_1, \cdots ,i_{k-1}\}} \Pi_{mI}\prod_{i\in I} x_i
-\sum_{I=\{i_1, \cdots , i_{k-1}\}:m\notin I}\Pi_{mI}\prod_{i\in I} x_i
\right],
\end{equation}
where
$\sum'$ denotes the sum over subsets $I=(i_1 \cdots i_{k-1})$ such that for each $l$ either 
$R_{i_l}< R_m$ or $R_{i_l}= R_m$ and $i_l < m$.
The corresponding system of ODEs becomes
 \begin{equation}
\label{eqdefgenmeanfpoolkthor4}
\dot x_m = \sum_{k=2}^K \ka x_m
\left[
k \sum'_{I=\{i_1 \cdots i_{k-1}\}} \Pi_{mI}\prod_{i\in I} x_i
-\sum_{I=\{i_1 \cdots i_{k-1}\}: m\notin I}\Pi_{mI}
\prod_{i\in I} x_i
\right],
\end{equation}
with $m=1, \cdots , d$.

The analog of Theorem \ref{th1} can now be easily given with the limiting deterministic dynamics being \eqref{eqdefgenmeanfpoolkthor4}.

It would be of course desirable to get some empirical data on the transition probabilities for $k$th order interactions.

\section{Introducing a forward-looking principal}
\label{secforwardprin1}

\subsection{Discrete time}
\label{secondiscrtimeforwardlooking}

Here we start exploiting another setting for the major player behavior. We shall assume that
changing strategies bears some costs, so that instantaneous adjustments of policies become unfeasible
and that the major player has some planning horizon with both running and (in case of a finite horizon) terminal costs.
For instance, running costs can reflect real spending and terminal cost some global objective,
like reducing the overall crime level by a specified amount. This setting will lead us to the class of problem that can be
called Markov decision (or control) processes (for the principal) on the evolutionary background (of
permanently varying profiles of small players).

We shall confine ourselves to the case of a finite-state-space of small players,
so that the state space of the group is given by vectors $x=(n_1, \cdots, n_d)/N$
 from the lattice $\Z^d_+/N$ (see Subsection \ref{secbestrespprin1}).
The extension to an arbitrary compact state-space is straightforward via Theorem
\ref{th2}.

Starting with a discrete time case,
we denote by $X_N(t,x,b)$ the Markov chain generated by \eqref{eqdefgenmeanfpool2}
 with a fixed $b$, that is by the operator
\begin{equation}
\label{eqdefgenmeanfpool2mod}
L_{b,N}f (x)=N\sum_{i,j: R_j(x,b)>R_i(x,b)} \ka x_i x_j
[R_j(x,b)-R_i(x,b)]\left[f\left(x-\frac{e_i}{N}+\frac{e_j}{N}\right)-f(x)\right],
\end{equation}
and starting in $x\in \Z^d_+/N$ at the initial time $t=0$.
We assume that the principal is updating her strategy
 in discrete times $\{k\tau\}$, $k=0,1, \cdots ..., n-1$, with some fixed $\tau>0$, $n\in \N$ aiming
 at finding a strategy $\pi$ maximizing the reward
\begin{equation}
\label{eqMarkdecevolback1}
V_n^{\pi,N}(x(N)) =\E_{N,x(N)}\left[\tau B(x_0,b_0)+\cdots +\tau B(x_{n-1},b_{n-1})+V_0(x_n)\right],
\end{equation}
where $B$ and $V_0$ are given functions (the running and the terminal payoff), $x_0=x(N)\in \Z^d_+/N$ also given,
\[
x_k=X_N(\tau,x_{k-1},b_{k-1}), \quad k=1,2, \cdots ,
\]
and $b_k=b_k(x_k)$ are specified by the strategy $\pi$ as some functions depending on the current state $x=x_k$
($\E_{N,x(N)}$ denotes the expectation specified by such process).
%One could consider more general strategies depending on the whole path-history, but in our Markov
%setting the optimal strategies are well known to be Markovian anyway, see e. g. \cite{HeLa}.
%As the state-space $\Z^d_+/N$ is discrete, no measurability assumption on functions $b_k(x)$
%is needed at this stage.
By the basic dynamic programming (see again \cite{HeLa}) the maximal rewards
$V_n^N(x(N))=\sup_{\pi} V_n^{\pi,N}(x(N))$
at different times $k$ are linked by the optimality equation
$V_k^N=S[N] V_{k-1}^N$,
where  the Shapley operator $S[N]$ (sometimes referred to as the Bellman operator)
is defined by the equation
\begin{equation}
\label{eqMarkdecevolback1a}
S[N]V(x)=\sup_b \left[\tau B(x,b)+\E V(X_N(\tau,x,b))\right],
\end{equation}
so that $V_n$
can be obtained by the $n$th iteration of the Shapley operator:
\begin{equation}
\label{eqMarkdecevolback1b}
V_n^N=S[N] V_{n-1}^N=S^n[N]V_0.
\end{equation}

Alternatively, in the infinite-horizon version, the principal can be interested in maximizing the discounted sum
\begin{equation}
\label{eqMarkdecevolback2}
V^{\pi,N}(x(N)) =\E_{N,x(N)}\sum_{k=0}^{\infty} \be^k B(x_k,b_k),
\end{equation}
with a $\be \in (0,1)$,
or any other criterion on the infinite horizon path.
Recall also that we assume $b$ to belong to a certain convex compact subset of a Euclidean space.

We are again interested in the law of large numbers limit $N\to \infty$, where we expect the limiting
problem for the principal to be the maximization of the reward
\begin{equation}
\label{eqMarkdecevolback3}
V_n^{\pi}(x_0) =\tau B(x_0,b_0)+\cdots +\tau B(x_{n-1},b_{n-1})+V_0(x_n),
\end{equation}
or respectively
\begin{equation}
\label{eqMarkdecevolback4}
V^{\pi}(x) =\sum_{k=0}^{\infty} \be^k \tau B(x_k,b_k)
\end{equation}
in the discounted infinite-horizon problem,
where
\begin{equation}
\label{eqMarkdecevolback5}
x_0=\lim_{\N\to \infty} x(N)
\end{equation}
(which is supposed to exist) and
\begin{equation}
\label{eqMarkdecevolback6}
x_k=X(\tau, x_{k-1}, b_{k-1}), \quad k=1,2, \cdots,
\end{equation}
with $X(t,x,b)$ denoting the solution to the characteristic system (or kinetic equations)
\begin{equation}
\label{eqMarkdecevolback7}
\dot x_j=\sum_{i} \ka x_i x_j[R_j(x,b)-R_i(x,b)], \quad j=1,...,d,
\end{equation}
with the initial condition $x$ at time $t=0$. Again by dynamic programming, the maximal reward in this problem
$V_n(x)=\sup_{\pi} V_n^{\pi} (x)$
is obtained by the iterations of the corresponding Shapley operator, $V_n =S^n V_0$, with
\begin{equation}
\label{eqMarkdecevolback8}
SV(x)=\sup_b \left[\tau B(x,b)+V(X(\tau,x,b))\right].
\end{equation}

Especially for the application to the continuous time models it is important to have estimates
of convergence uniform in $n=t/\tau$ for bounded total time $t=n\tau$.

\begin{theorem}
\label{thMarkDeconEvolBack1}
(i) Assume the functions $R_i(x,b)$ and $B(x,b)$ belong to $C_{bLip}(\Si_d)$ as the functions of the first variable
with the norm uniformly bounded with respect to the second variable.
 Assume also \eqref{eqMarkdecevolback5} holds.
Then, for any Lipschitz function $V_0$ on $\Si_d$, $\tau>0$ and $n\in N$,
\begin{equation}
\label{eq1thMarkDeconEvolBack1}
|V_n^N(x(N))-V_n(x)|\le C(t)(|x(N)-x|+t^{2/3}(n/N)^{1/3}) \|V_0\|_{bLip},
\end{equation}
where $t=n\tau$ is the total time. In particular, for $n=N^{\om}$ with $\om \in (0,1)$, the last term on the r.h.s. of
\eqref{eq1thMarkDeconEvolBack1} becomes of order $N^{-(1-\om)/3}$.

(ii) If there exists a Lipshitz continuous optimal policy $\pi=\{b_k(x)\}$, $k=1, \cdots , n$,
 for the limiting optimization problem,
then $\pi$ is approximately optimal for the $N$-agent problem, in the sense that for any $\ep>0$ there exists
$N_0$ such that, for all $N>N_0$,
\[
|V_n^N(x(N)) -V_n^{N,\pi}(x(N))|\le \ep.
\]
\end{theorem}

\begin{proof}
(i) Assume $V_0$ is Lipschitz with
some Lipschitz constant $\ka$. This implies that all functions $V_k(x)$ are uniformly Lipschitz continuous. In fact,
 \[
| SV(x_1)-SV(x_2)| \le \sup_b |\tau B(x_1,b)+V(X(\tau,x_1,b))-\tau B(x_2,b)-V(X(\tau,x_2,b))|
\]
\[
\le \ka_B \tau |x_1-x_2| + \ka e^{\tau F} |x_1-x_2|,
\]
where $\ka_B$ is the Lipschitz constant for $B$ and $F$ is the Lipschitz constant of the function
on the r.h.s. of \eqref{eqMarkdecevolback7} (as a functions of $x$). Thus the Lipschitz constant of $V_k=S^kV_0$
is bounded by a constant $C(t)$.
Notice also that, since the function $B$ is uniformly bounded, all $V_k^N$ and $V_k$ are uniformly bounded, say by some constant $v$.

 Next we can write
\[
S^n[N]V_0-S^nV_0=\sum_{j=0}^{n-1} S^j[N](S[N]-S)S^{n-(j-1)}[N]V_0.
\]
Consequently,
\[
\|S^n[N]V_0-S^nV_0\|\le n \sup_{k=1, \cdot ,n} \|(S[N]-S)S^k V_0\|.
\]
Since the uniform estimate of the difference of two functions of $b$ implies the same estimate for the
difference of the maxima, it follows from Theorem \ref{th1} that
\[
\|S^n[N]V_0-S^nV_0\|\le n C(t) \tau ^{2/3}N^{-1/3} \|V_0\|_{bLip}.
\]
yielding \eqref{eq1thMarkDeconEvolBack1}.

(ii) One shows as above that for any Lipschitz continuous policy $\pi$,
the corresponding value functions $V^{\pi,N}$ converge. Combined with (i),
this yields Statement (ii).
\end{proof}

\begin{remark} For a compact state space being a subset of $\R^n$ one would get for the last term
of the r.h.s. of \eqref{eq1thMarkDeconEvolBack1} the decay estimate of order $t^{1-1/(2n+3)} (n/N)^{1/(2n+3)}$.
\end{remark}

Since the tails of series \eqref{eqMarkdecevolback4} and \eqref{eqMarkdecevolback2} tend to zero uniformly,
the following fact is a consequence of Theorem \ref{thMarkDeconEvolBack1}.

\begin{theorem}
\label{thMarkDeconEvolBack2}
Under the assumptions of Theorem \ref{thMarkDeconEvolBack1} the discounted optimal rewards
\eqref{eqMarkdecevolback2} converge, as $N\to \infty$, to the discounted reward \eqref{eqMarkdecevolback4}.
\end{theorem}

Analyzing long time behavior of the optimal dynamics given by  Theorem \ref{thMarkDeconEvolBack1}
leads one naturally to the analysis of the fixed points of equation \eqref{eqMarkdecevolback7}
and their turnpike properties. Namely, let $X[b]$ denote the set of fixed points of
\eqref{eqMarkdecevolback7} for  given $b$. If
\begin{equation}
\label{eqMarkEvolturnp}
\sup B(x,b) =\max_b \max_{x\in X[b]}B(X[b],b),
\end{equation}
the points of maximum on the r. h. s. can be expected to serve as turnpikes (introduced in economics by \cite{DoSaSo},
see recent reviews e. g. in \cite{Zas06} and \cite{KoWei12})
for long time behavior of optimal problems arising from the limiting evolution  of \eqref{eqMarkdecevolback7}.
How this fact is recast in terms of the Markov decision process with $N$ players is an interesting
problem for what one can characterize as the turnpike theory for Markov control on evolutionary background.
We shall not touch it here.

\subsection{Continuous time}

Here we initiate the analysis of the optimization problem for a forward-looking principal
in continuous time choosing the most transparent deterministic evolution of the principal.
Namely, let the efforts (budget) $b$ of the major player evolve according to the equation
$\dot b=u$
with control $u$ from a compact convex set $U\in \R^r$.
The state space of the group being again given by vectors $x=(n_1, \cdots, n_d)/N$ from the lattice $\Z^d_+/N$,
the payoff of the major player will be given by
\[
\int_t^T J(x(s),(b(s), u(s)) \, ds +S_T(x(T),(b(T))
\]
where $J,S_T$ are some continuous functions uniformly Lipschitz in all their variables. The optimal payoff of the major player is thus
\begin{equation}
\label{eqHJBmaj0}
S_N(t,x(N),b)=\sup_{u(.)\in \tilde U}\E^N_{x(N),b}  \left\{ \int_t^T J(x(s),(b(s),u(s)) \, ds +S_T(x(T),(b(T))\right\},
\end{equation}
where $\E^N_{x,b}$ is the expectation of the corresponding Markov process starting at the position $(x,b)$ at time $t$,
and $\tilde U$ is the class of controls that are piecewise constant in $t$
 and Lipschitz in $x,b$ (so that the equations $\dot b =u(x,b)$ are trivially well-posed).
We are now in the standard Markov decision setting of a controlled Markov process generated by
the operator $L_{b,N}$ from \eqref{eqdefgenmeanfpool2}, or more precisely
\[
L_{b,N}S_N(t,x,b)=N\sum_{i,j: R_j(x,b)>R_i(x,b)} \ka x_i x_j
\]
\begin{equation}
\label{eqdefgenforlook}
\times
[R_j(x,b)-R_i(x,b)][S_N(t, x-e_i/N+e_j/N,b)-S_N(t,x,b)].
\end{equation}
As $N\to \infty$, the dimension of vectors $x$ tends to infinity making direct calculations complicated.

As seen from  \eqref{eqdefgenmeanfpool5}, the operators $L_{b,N}$ tend to a simple first order PDO, so that the limiting
optimization problem of the major player turns out to be the problem of finding
\begin{equation}
\label{eqlimoptimmajorconttime}
S(t,x,b)=\sup_{u(.)\in \tilde U} \left\{ \int_t^T J(x(s),b(s),u(s)) ds +S_T(x(T),(b(T))\right\},
\end{equation}
where $(x(s),(b(s))$ (depending on $u(.)$) solve the system of equations $\dot b=u$ and
\[
\dot x_j=\sum_{i} \ka x_i x_j[R_j(x,b)-R_i(x,b)], \quad j=1,...,d.
\]
The well-posedness of this system is a straightforward extension of the well-posedness of equations
\eqref{eqdefgenmeanfpool6brep}.

Instead of proving the convergence $ S_N(t,x(N),b) \to S(t,x,b)$, we shall concentrate on a more
practical issue comparing the corresponding discrete time approximations, as these approximations are
usually exploited for practical calculations of $S_N$ or $S$.

The discrete-time approximation to the limiting problem of finding \eqref{eqlimoptimmajorconttime}
is the problem of finding
\begin{equation}
\label{eqMarkdecevolcontdisc1}
V_{t,n}(x,b)=\sup_{\pi} V_{t,n}^{\pi}(x,b) =\sup_{\pi}\left[\tau J(x_0,b_0,u_0)+\cdots +\tau J(x_{n-1},b_{n-1}, u_{n-1})+V_0(x_n,b_n)\right],
\end{equation}
where $\tau =(T-t)/n$, $(x_0,b_0)=(x,b)$, $V_0(x,b)=S_T(x,b)$ and
\begin{equation}
\label{eqMarkdecevolcontdisc2}
b_k=b_{k-1}+u_{k-1}\tau, \quad
x_k=X(\tau, x_{k-1}, b_{k-1}), \quad k=1,2, \cdots,
\end{equation}
with $X(t,x,b)$ solving equation
\eqref{eqMarkdecevolback7}
with the initial condition $x$ at time $t=0$.
The discrete-time approximation to the initial optimization problem
is the problem of finding
\[
V_{t,n}^N(x_0,b_0)=\sup_{\pi} V_{t,n}^{\pi,N}(x_0,b_0)
\]
\begin{equation}
\label{eqMarkdecevolcontdisc3}
 =\sup_{\pi}\E_{N,x(N),b}
\left[\tau J(x_0,b_0,u_0)+\cdots +\tau J(x_{n-1},b_{n-1}, u_{n-1})+V_0(x_n,b_n)\right],
\end{equation}
where $x_k=X_N(\tau,x_{k-1},b_{k-1})$ with $X_N(t,x,b)$ denoting the Markov process with generator
\eqref{eqdefgenmeanfpool2mod}. The strategies $\pi$ here specify the choice of control parameters
$u_k$ based on the previous information.

\begin{remark}
It is well known that $V_n(x,b)$ and $V_n^N(x,b)$ with $V_0=S_T$ approach
the optimal solutions $S(T-t,b,x)$ and $S_N(T-t,x,b)$
given by \eqref{eqlimoptimmajorconttime} and \eqref{eqHJBmaj0} respectively, see e. g. Theorem 4.1 of
\cite{FleSo} or Theorem 3.4 of \cite{KolMas}.
\end{remark}

\begin{theorem}
\label{thMarkDeconEvolcontdisc}
Recall that $J,S_T$ are uniformly Lipschitz in all their variables. Then,
for any $x$ and $t\in [0,T]$
\begin{equation}
\label{eq1thMarkDeconEvolcontdisc}
|V_{t,n}^N(x)-V_{t,n}(x)|\le C(T) (T-t)^{2/3} (n/N)^{1/3} \|V_0\|_{bLip}.
\end{equation}
\end{theorem}

\begin{proof}
This is a direct consequence of Theorem \ref{thMarkDeconEvolBack1}. The only difference is the use of control parameter
that is distinct from the state $b$, but this does not affect the proof.
\end{proof}

\section{Models of growth under pressure}
\label{secmodgrowthpres}

\subsection{General convergence result for evolutions in $l^1$}
\label{secconvcount}

Here we extend the results of Subsection \ref{secondiscrtimeforwardlooking} in two directions,
namely, by working with a countable (rather than finite or compact) state-space and unbounded rates,
and with more general interactions allowing in particular for a change in the number of particles.

Thus we take the set of natural numbers $\{1,2, \cdots \}$ as the state space of each small player, the set of finite
Borel measures on it being the Banach space $l^1$ of sumable real sequences $x=(x_1, x_2, \cdots ) $.

Thus the state space of the total multitude of
small players will be formed by the set $\Z^{fin}_+$ of sequences of integers
$n=(n_1, n_2, \cdots )$ with only finite number of non-vanishing ones,
with $n_k$ denoting the number of players in the state $k$, the total number
of small players being $N=\sum_k n_k$.
As we are going to extend the analysis to processes not preserving the number of particles,
we shall work now with a more general scaling
of the states, namely with the sequences
\[
x=(x_1,x_2, \cdots ...)=hn =h(n_1,n_2, \cdots ...)\in h\Z_+^{fin}
\]
with certain parameter $h>0$, which can be taken, for instance, as the inverse number to
the total number of players $\sum_k n_k$ at the initial moment of observation.
The necessity to distinguish initial moment is crucial here, as this number changes over time.
Working with the scaling related to the current number of particles $N$ may lead, of course,
to different evolutions.

The general processes of birth, death, mutations and binary interactions that can occur under an
influence $b$ of the principle are Markov chains on $h\Z_+^{fin}$ specified by the generators of the following type
\[
L_{b,h}F(x)=\frac{1}{h}\sum_j \be_j(x,b)[F(x+he_j)-F(x)]
+\frac{1}{h}\sum_j \al_j(x,b)[F(x-he_j)-F(x)]
\]
\[
+\frac{1}{h}\sum_{i,j} \al_{ij}^1(x,b)[F(x-he_i+he_j)-F(x)]
+\frac{1}{h}\sum_{i,(j_1,j_2)} \al_{i(j_1j_2)}^1(x,b)[F(x-he_i+he_{j_1}+he_{j_2})-F(x)]
\]
\[
+\frac{1}{h}\sum_{(i_1,i_2),j} \al^2_{(i_1i_2)j}(x,b)[F(x-he_{i_1}-he_{i_2}+he_j)-F(x)]
\]
\begin{equation}
\label{eqgenbininterpres}
+\frac{1}{h}\sum_{(i_1,i_2)}\sum_{(j_1,j_2)} \al^2_{(i_1i_2)(j_1j_2)}(x,b)[F(x-he_{i_1}-he_{i_2}+he_{j_1}+he_{j_2})-F(x)],
\end{equation}
where brackets $(i,j)$ denote the pairs of states.
Here the terms with $\be_j$ and $\al_j$ describe the spontaneous injection (birth) and death of agents,
the terms with $\al^1$ describe the multiplication or mutations of single agents (including fragmentation and splitting),
the terms with $\al^2$ describe the binary interactions, with all terms including possible mean-field interactions.
Say, our model \eqref{eqdefgenmeanfpool2}
was an example of binary interaction.

Let $L$ be a positive increasing function on $\N$ such that $L(j)\to \infty$ as $j\to \infty$.
We shall refer to such functions as Lyapunov functions.
Notations from Appendix B will be used here for different norms and notions related to a Lyapunov function $L$
(see \eqref{eqdefweighnorml1} and the discussion around it).
We say that the generator $L_{b,h}$ with $\be_j=0$ and the corresponding process do not increase $L$ if for any allowed
transition the total value of $L$ cannot increase, that is if $\al_{ij}^1 \neq 0$, then $L(j)\le L(i)$,
if  $\al_{i(j_1,j_2)}^1 \neq 0$, then $L(j_1)+L(j_2)\le L(i)$,  if $\al_{(i_1i_2)j}^2 \neq 0$, then $L(j)\le L(i_1)+L(i_2)$,
if  $\al_{(i_1i_2)(j_1j_2)}^2 \neq 0$, then $L(j_1)+L(j_2)\le L(i_1)+L(i_2)$.
If this is the case, then the chains generated by $L_{b,h}$ always remain in a ball $B_+(L,R)$,
if they were started there. Hence for any $h$ and $R$, $L_{b,h}$ generates a well-defined Markov chains
$X_{b,h}(t,x)$ in any of the finite state-spaces $h\Z_+^{fin}\cap B_+(L,R)$
(the corresponding Kolmogorov $Q$-matrices are transpose to the matrices representing $L_{b,h}$).

A generator $L_{b,h}$ is called
$L$-subcritical if $L_{b,h}(L) \le 0$. Of course, if $L_{b,h}$ does not increase $L$, then it is $L$-subcritical.
Though the condition to not increase $L$ seems to be restrictive, many concrete models satisfy it,
for instance the celebrated merging-splitting (Smoluchovskii) process considered below. On the other hand,
models with spontaneous injections may increase $L$, so that one is confined to work with the weaker
property of sub-criticality.

We shall denote by $C(B_+(L,R)\subset l^1)$ and $C^k(B_+(L,R)\subset l^1)$ the spaces of
continuous and differentiable functions on $B_+(L,R)$ with $B_+(L,R)$ considered as a subset of $l^1$,
that is, equipped with the topology of $l^1$, where these sets are easily seen to be compact.
Similar notations for Banach-space valued functions will be used.

By Taylor-expanding $F$ in \eqref{eqgenbininterpres} one sees that if $F$ is sufficiently smooth,
the sequence $L_{b,h}F$ converges to
\[
\La_bF(x)=\sum_j (\be_j(x,b)-\al_j(x,b))\frac{\pa F}{\pa x_i}
+\sum_{i,j} \al_{ij}^1(x,b)[\frac{\pa F}{\pa x_j}-\frac{\pa F}{\pa x_i}]
\]
\[
+\sum_{i,(j_1,j_2)} \al_{i(j_1j_2)}^1(x,b)[\frac{\pa F}{\pa x_{j_1}} +\frac{\pa F}{\pa x_{j_2}}-\frac{\pa F}{\pa x_i}]
+\sum_{(i_1,i_2)}\sum_j \al^2_{(i_1i_2)j}(x,b)[\frac{\pa F}{\pa x_j}-\frac{\pa F}{\pa x_{i_1}}-\frac{\pa F}{\pa x_{i_2}}]
\]
\begin{equation}
\label{eqgenbininterpreslim}
+\sum_{(i_1,i_2)}\sum_{(j_1,j_2)} \al^2_{(i_1i_2)(j_1j_2)}(x,b)
[\frac{\pa F}{\pa x_{j_1}} +\frac{\pa F}{\pa x_{j_2}}-\frac{\pa F}{\pa x_{i_1}}-\frac{\pa F}{\pa x_{i_2}}].
\end{equation}
Moreover,
\begin{equation}
\label{eqgenbininterpreslim1}
\|(L_{b,h}-\La_b)F\|_{C(B_+(L,R)\subset l^1)}
 \le 8h \ka (L,R) \|F\|_{C^2(B_+(L,R)\subset l^1)},
\end{equation}
with $\ka (L,R)$ being the $\sup_b$ of the norms
\[
\left\|\sum_i (\al_i+\be_i) +\sum_{i,j} \al_{ij}^1+\sum_{i,(j_1,j_2)} \al^1_{i(j_1j_2)}
+\sum_{(i_1,i_2),j} \al^2_{(i_1i_2)j}
+\sum_{(i_1,i_2),(j_1,j_2)} \al^2_{(i_1i_2)(j_1j_2)}\right\|_{C(B_+(L,R)\subset l^1)}.
\]

By regrouping the terms of $\La_b$, it can be rewritten in the form of the general first order operator
\begin{equation}
\label{eqgenbininterpreslim2}
\La_b F(x)=\sum_j f_j(x)\frac{\pa F}{\pa x_j},
\end{equation}
where
\[
f_i=\be_i -\al_i+\sum_k (\al^1_{ki}-\al^1_{ik})
+\sum_k [\al^1_{k(ii)}+\sum_{j\neq i} (\al^1_{k(ij)}+\al^1_{k(ji)}]-\sum_{(j_1,j_2)} \al^1_{i(j_1j_2)}
\]
\[
+\sum_{(j_1,j_2)} \al^2_{(j_1j_2)i}-\sum_k [\al^2_{(ii)k}+\sum_{j\neq i} (\al^2_{(ij)k}+\al^2_{(ji)k})]
\]
\[
+\sum_{(j_1,j_2)}[\al^2_{(j_1j_2)(ii)}+\sum_{j\neq i} (\al^2_{(j_1j_2)(ji)}+\al^2_{(j_1j_2)(ij)})]
\]
\[
-\sum_{(j_1,j_2)}[\al^2_{(ii)(j_1j_2)}+\sum_{j\neq i} (\al^2_{(ij)(j_1j_2)}+\al^2_{(ji)(j_1j_2)})].
\]

Its characteristics solving the ODE $\dot x =f(x)$ can be expected to describe the limiting behavior
of the Markov chains $X_{b,h}(x,t)$ for $h\to 0$.

\begin{theorem}
\label{th1infin}
Assume the operators $L_{b,h}$ are $L$ non-increasing for
a Lyapunov function $L$ on $\Z$ such that $L(j)\to \infty$ as $j\to \infty$,
the function $f: l^1_+ \to l^1$ is uniformly Lipschitz on $B_+(R,L)$ and $\ka(\La,R)<\infty$.
Then the Markov chains $X_h(t,x(h))$ with $x(h)\in B_+(R,L)$ converge in distribution to the deterministic
evolution $X(t,x)$ solving equation $\dot x=f(x)$ and moreover
\begin{equation}
\label{eq1ath1infin}
|\E F (X_h(t,x(h)))-F(X(t,x(h)))| \le t C(R,t)\frac{1}{N^{1/3}} \|F \|_{C^2(B_+(L,R)\subset l^1)}
\end{equation}
\begin{equation}
\label{eq1bth1infin}
|\E F (X_h(t,x(h)))-F(X(t,x(h)))| \le C(R,t)\frac{t^{4/5}}{N^{1/5}} \|F \|_{C_{bLip}(B_+(L,R)\subset l^1)}
\end{equation}
with constants $C(R,t)$.
If $f$ is uniformly twice continuously differentiable, then the same estimates hold with the improved rates
$1/N$ and $1/N^{1/3}$ respectively.
\end{theorem}

\begin{proof}
The proof is similar to the proof of Theorem \ref{th2},
though the lack of compactness is dealt with by $L$-subcritical condition that again
allows one to use effective finite-dimensional approximations.  Moreover, discrete setting allows one not to bother
about weak topology.

If $f$ is smooth and
\[
\|f\|_{C^2(B_+(L,R)\subset l^1;l^1)} \le D(R),
\]
then the solutions $X(t,x)$ to the equation $\dot x =f(x)$ are twice differentiable with respect to initial data and the
corresponding mapping $U^t:F(x) \mapsto F(X(t,x))$ are twice continuously differentiable by Lemma \eqref{lemonsenseBanach1}.
Hence the estimate
\begin{equation}
\label{eq1ath1infinrep}
|\E F (X_h(t,x(h)))-F(X(t,x(h)))| \le t C(R,t)\frac{1}{N} \|F \|_{C^2(B_+(L,R)\subset l^1)},
\end{equation}
claimed by the last statement of the Theorem, follows directly by \eqref{eqcompsem2} and \eqref{eqgenbininterpreslim1}.

If $f$ is only Lipschitz continuous we again use a finite-dimensional approximation $F(x) \to F_j(x)=F(P_j^*(x))$,
where now $P_j^*$ is just the projection on the first $j$ coordinates, that is $[P^*_j(x)]_k=x_k$
for $k\le j$ and $[P_j^*(x)]_k=0$ otherwise. For $x\in B_+(L,R)$,
\[
\|P_j^*(x)-x\|_{l^1}\le \frac{R}{L(j)},
\]
and hence one can further use the smooth approximation $\Phi_{\de}(F\circ P_j^*)$
with the same effect as in Theorem \ref{th2}. The dimension of the image of $P_j^*$ is $j$,
so the results of Theorem \ref{th2} apply with $n=1$ yielding \eqref{eq1ath1infin} and \eqref{eq1bth1infin}.
\end{proof}

Assume now that the principal is updating her strategy
 in discrete times $\{k\tau\}$, $k=0,1, \cdots ..., n-1$, with some fixed $\tau>0$, $n\in \N$ aiming
 at finding a strategy $\pi$ maximizing the reward \eqref{eqMarkdecevolback1}, but now with
 $x_0=x(h)\in h\Z^{fin}\cap  B_+(L,R)$.
 Using  Theorem \ref{th1infin},
 It is straightforward to extend  Theorem \ref{thMarkDeconEvolBack1} to the present
  setting of a countable state-space. Using the same notations as in  Theorem \ref{thMarkDeconEvolBack1}
  for rewards and Shapley operators yields the following result.

\begin{theorem}
\label{thMarkDisccontEvolu}

Assume the conditions of Theorem \ref{th1infin} hold and the function $B(x,b)$ is
uniformly Lipschitz on $B_+(R,L)$ as a function of the first variable.
Then, for any continuous $V_0$ on $B_+(R,L)$, $\tau>0$ and $n\in N$,
\begin{equation}
\label{MarkDeconEvolBackconttime}
|V_n^N(x(N))-V_n(x)|\le C(t)(|x(N)-x|+t^{4/5}(n/N)^{1/5}) \|V_0\|_{bLip},
\end{equation}
where $t=n\tau$ is the total time. In case when $f$ and $B$ are twice continuously differrentiable,
the rates of convergence improve to $N^{-1/3}$.
\end{theorem}

\subsection{Evolutionary coalition building under pressure}
\label{seccoalbuildpres}

As a direct application of Theorem \ref{thMarkDisccontEvolu}, let us
discuss the model of evolutionary coalition building.
Namely, so far we talked about small players that occasionally and randomly
exchange information in small groups (mostly in randomly formed pairs) resulting in copying
the most successful strategy by the members of the group. Another natural reaction of the society
of small players to the pressure exerted by the principal can be executed by forming stable groups that
can confront this pressure in a more effective manner (but possibly imposing certain obligatory regulations
for the members of the group). Analysis of such possibility leads one naturally to models of mean-field-enhanced coagulation
processes under external pressure.
Coagulation-fragmentation processes are well studied in statistical physics, see e. g. \cite{Nor00}.
In particular, general mass-exchange processes, that in our social environment become general coalition forming processes
preserving the total number of participants, were analyzed in \cite{Ko04} and
\cite{Ko06} with their law of large number limits for discrete and general
state spaces. Here we add to this analysis a strategic framework for a major player fitting the model
to the more general framework of the previous section.
Instead of coagulation and fragmentation we shall use here the terms merging and splitting or breakage.

For simplicity, we ignore here any other behavioral distinctions
(assuming no strategic space for an individual player)
concentrating only on the process of forming coalitions.
Thus the state space of the total multitude of
small players will be formed by the set $\Z^{fin}_+$ of sequences of integers
$n=(n_1, n_2, \cdots ...)$ with only finite number of non-vanishing ones,
with $n_k$ denoting the number of coalition of size $k$, the total number
of small players being $N=\sum_k kn_k$ and the total number of coalitions
(a single player is considered to represent a coalition of size $1$) being $\sum_k n_k$.
 Also for simplicity we reduce attention
to binary merging and breakage only, extension to arbitrary regrouping processes
from \cite{Ko04} (preserving the number of players) is more-or-less straightforward.

 As previously, we will look for the evolution of appropriately scaled states, namely the sequences
\[
x=(x_1,x_2, \cdots ...)=hn =h(n_1,n_2, \cdots ...)\in h\Z_+^{fin}
\]
with certain parameter $h>0$, which can be taken, for instance, as the inverse number to
the total number of coalitions $\sum_k n_k$ at the initial moment of observation.

If any randomly chosen pair of coalitions of sizes $j$ and $k$ can merge with the rates $C_{kj}(x,b)$,
which may depend on the whole composition $x$ and the control parameter $b$ of the major player,
and any randomly chosen coalition of size $j$ can split (break, fragment) into two groups of sizes $k<j$ and $j-k$
with rate $F_{jk}(x,b)$, the limiting deterministic evolution of the state is known to be described by the
system of the so-called Smoluchovski equations
\begin{equation}
\label{eqSmoleqcoagfrag}
\dot x_k=f_k(x)=\sum_{j<k} C_{j,k-j}(x,b) x_j x_{k-j}-2\sum_j C_{kj}(x,b) x_j x_k
+2\sum_{j>k} F_{jk}(x,b) x_j -\sum_{j<k} F_{kj}(x,b) x_k.
\end{equation}
In addition to the well known setting with constant $C_{jk}$ and $F_{jk}$ (see e. g. \cite{BaCa}) we added here the
mean field dependence of these coefficients (dependence on $x$) and the dependence on the control parameter $b$.

As one easily checks, equations \eqref{eqSmoleqcoagfrag} can be written in the equivalent weak form
\begin{equation}
\label{eqSmoleqcoagfragweak}
\frac{d}{dt} \sum_j g_jx_j=\sum_{j,k}(g_{j+k}-g_j-g_k) C_{jk}(x,b) x_j x_k
+\sum_j \sum_{k<j} (g_{j-k}+g_k-g_j) F_{jk}(x,b) x_j,
\end{equation}
which should hold for a suitable class of test functions $g$. For instance, under the assumption of bounded coefficients
(see \eqref{eqSmoleqcoagfragboundco} below), the class of test functions is the class of all functions from
$l^{\infty}=\{g: \sup_j |g_j|<\infty\}$.
This implies, in particular, that the corresponding semigroups \eqref{eqdefflowonkineq}
on the space of continuous functions,
that is $U^tG(x)=G(X(t,x))$, have the generator
\[
\La_b G(x)=\sum_k f_k(x) \frac{\pa G}{\pa x_k}(x)
=\sum_{j,k}\left(\frac{\pa G}{\pa x_{k+j}}-\frac{\pa G}{\pa x_j}-\frac{\pa G}{\pa x_k}\right) C_{jk}(x,b) x_j x_k
\]
\begin{equation}
\label{eqencoagfrag}
+\sum_j \sum_{k<j} \left(\frac{\pa G}{\pa x_{j-k}}-\frac{\pa G}{\pa x_j}+\frac{\pa G}{\pa x_k}\right) F_{jk}(x,b) x_j
\end{equation}
of type \eqref{eqgenbininterpreslim}.

Let $R_j(x,b)$ be the payoff for the member of a coalition of size $j$. In our strategic setting,
the rates $C_{jk}(x,b)$ and $F_{jk}(x,b)$ should depend on the differences of these rewards before and after merging or splitting.
For instance, the simplest choices can be
\begin{equation}
\label{eqcoagstrateg1}
C_{kj}(x,b)=a_{j+k,k}\1_{R_{k+j}\ge R_k}  (R_{k+j}- R_k)
+a_{j+k,j}\1_{R_{k+j}\ge R_j}  (R_{k+j}- R_j),
\end{equation}
with some constants $a_{lk}\ge 0$
reflecting the assumption that merging may occur whenever it is beneficial for all members concerned but weighted according
to the size of the coalitions involved, where by $\1_M$ here and in what follows we denote the indicator function of the set $M$.
Similarly
\begin{equation}
\label{eqfragstrateg1}
F_{kj}(x,b)=\tilde a_{kj}\1_{R_j \ge R_k} (R_j-R_k)+\tilde a_{k,k-j}\1_{R_{k-j} \ge R_k} (R_{k-j}-R_k).
\end{equation}

A Markov approximation to dynamics \eqref{eqSmoleqcoagfrag} is constructed in the standard way,
which is analogous to the constructions of approximating Markov chains described in the previous section
(for coagulation - fragmentation processes this Markov approximation is often referred
 to as the Markus-Lushnikov process, see e.g. \cite{Nor00}), namely, by attaching exponential
clocks to any pair of coalitions that can merge with rates $C_{kj}$ and to any coalition
that can split with rates $F_{kj}$. This leads to a Markov chain $X_h(t,x,b)$ on $h\Z^{fin}_+$ with the generator
\[
\La_{b,h}G (x)=\sum_{i,j} C_{ij}(x,b) x_i x_j [G(x-he_i-he_j+he_{i+j})-G(x)]
\]
\begin{equation}
\label{eqSmolgencoagfrag}
+\sum_i \sum_{j<i} F_{ij}(x,b) x_i [G(x-he_i+he_j+he_{i+j})-G(x)],
\end{equation}
where $e_1,e_2, \cdots$ denote the standard basis in $\R^{\infty}$.
There exists an extensive literature showing the well -posedness of infinite-dimensional dynamics
\eqref{eqSmoleqcoagfrag} and proving the convergence, as $h\to 0$, of Markov chains generated by \eqref{eqSmolgencoagfrag}
under various assumptions on the coefficients $C$ and $F$ (see e. g. \cite{Nor00} and \cite{Ko06} and references therein).
However, to deal with a forward -looking principal, some uniform rates of convergence are needed, like those
of Theorem \ref{th1infin}.

We shall propose here only the simplest result in this direction assuming that
the intensities of individual transition are uniformly bounded and uniformly Lipschitz,
 that is
\begin{equation}
\label{eqSmoleqcoagfragboundco}
C=\sup_{j,k} C_{jk}(x,b) <\infty, \quad F =\sup_j \sum_{k<j} F_{kj}(x,b) <\infty,
\end{equation}
\begin{equation}
\label{eqSmoleqcoagfragboundco2}
\begin{aligned}
& C(1)=\sup_{b,j,k} \|C_{jk} (.,b)\|_{C_{bLip}(B_+(R,L)\subset l^1)} <\infty, \\
& F(1) =\sup_{b,j} \sum_{k<j} \|F_{kj}(.,b)\|_{C_{bLip}(B_+(R,L)\subset l^1)} <\infty.
\end{aligned}
\end{equation}
Notice however that the overall intensities are still unbounded (quadratic),
so that we are still quite away from the assumptions of Section \ref{secforwardprin1}.

Choosing the function $L(j)=j$ we see that Markov chains $X_h(t,x,b)$ do not increase $L$.
Moreover,  \eqref{eqSmoleqcoagfragboundco} implies
\begin{equation}
\label{eqSmoleqcoagfragboundco1}
\begin{aligned}
& \sup_b \|f (.,b)\|_{C(B_+(R,L)\subset l^1);l^1} \le 3CR^2+3FR, \\
& \sup_b \|f (.,b)\|_{C_{bLip}(B_+(R,L)\subset l^1); l^1} \le 6CR+3F+3(C(1)R+F(1))R
\end{aligned}
\end{equation}
and hence the following result.

\begin{theorem}
\label{thMarkDisccontEvolucoal}
For a model of strategically enhanced coalition building subject to \eqref{eqSmoleqcoagfragboundco} and \eqref{eqSmoleqcoagfragboundco2}
the conditions of Theorem \ref{th1infin} and consequently its assertions are satisfied.
\end{theorem}

\subsection{Strategically enhanced preferential attachment on evolutionary background}
\label{secbirthdeath}

A natural and useful extension of the theory presented above can be obtained by the
inclusion in our pressure-resistance evolutionary-type game the well known model
of linear growth with preferential attachment (Yule, Simon and others,
see  \cite{SimRoy} for review) turning the latter into a strategically enhanced preferential attachment model
that includes evolutionary-type interactions between agents and a major player having tools
to control (interfere into) this interaction.
Since the proper exposition of the corresponding rigorous convergence result requires
an extension of Theorem  \ref{thMarkDisccontEvolu} to $L$-subcritical (rather than $L$-non-increasing)
processes, we shall not present it here, but only indicate the expected outcomes
leaving  details to another publication.

We shall work with the general framework of Theorem  \ref{thMarkDisccontEvolu}, having in mind that the basic
examples of the approximating Markov chains $X_h(t,x(h))$ can arise from the merging and splitting coalition model
of the previous section (with generator \eqref{eqSmolgencoagfrag})
or from setting \eqref{eqdefgenmeanfpool2mod}, where now the number of possible states
$j$ becomes infinite and hence, assuming for simplicity that the agents are identical so that the parameter $j$
denotes the size of the coalition, generator \eqref{eqdefgenmeanfpool2mod} becomes
\begin{equation}
\label{eqdefgenmeanfpool2modrep}
L_{b,h}G (x)=\frac{1}{h}\sum_{i,j: R_j(x,b)>R_i(x,b)} \ka x_i x_j
[R_j(x,b)-R_i(x,b)][G\left(x-h e_i+h e_j\right)-G(x)],
\end{equation}
where $R_j(x,b)$ is the payoff to a member of a coalition of size $j=1,2, \cdots$. The Markov chain with generator
\eqref{eqdefgenmeanfpool2modrep} describes the process where agents  can move from one coalition to another
choosing the size of the coalition that is more beneficial under the control $b$ of the principal.
Of course one can work also with various combinations of generators \eqref{eqdefgenmeanfpool2modrep} and
$\La_{b,h}$ from \eqref{eqSmolgencoagfrag}, as well as with their various extensions including, say, $k$th order interactions,
see \eqref{eqdefgenmeanfpoolkthor1}, or various classes (for instance, levels of activity) of agents,
where coalitions get another interpretation as groups of agents following certain particular strategy.

The most studied form of preferential attachment evolves by the discrete time injections
of agents (see \cite{BaAl99}, \cite{DeMo13}, \cite{SimRoy} and references therein).
Along these lines, we can assume that with time intervals $\tau$ a new agent enters the system
in such a way that with some probability $\al(x,b)$ (which, unlike the standard model, can now depend
on the distribution $x$ and the control parameter $b$ of the principal) she does not enter
any of the existing coalitions (thus forming a new coalition of size $1$), and with probability $1-\al(x,b)$
she joins one of the coalitions, the probability to join a coalition being proportional to its size
(this reflects the notion of preferential attachment coined in \cite{BaAl99}).
Thus if $V(x)$ is some function on the state space $h\Z^{fin}_+$, its expected
value after a single entry changing $x$ to $\hat x$ is descried by the following operator $T_h$:
\begin{equation}
\label{eqpreferattachdiscr}
T_hV(x)=\E V(\hat x)= \al V(x+he_1)+(1-\al) \sum_{k=1}^{\infty}\frac{kn_k}{L(n)}V(x-he_k+he_{k+1}),
\end{equation}
where $L(n)=\sum kn_k$, $x=nh$.

A continuous time version of these evolutions can be modeled by a Markov process, where the injection occurs with
some intensity $\la (x,b)$ (that can be influenced by the principal subject to certain costs).
In other words, it can be included by adding to generator
\eqref{eqdefgenmeanfpool2modrep} or \eqref{eqSmolgencoagfrag}
the additional term of the type
\[
\La_{b,h}^{att}G(x)=\frac{\al \la (b,x)}{h}[G(x+he_1)-G(x)]
+ \frac{(1-\al)\la (b,x)}{h}\sum_{k=1}^{\infty} k x_k[G(x-h e_k+h e_{k+1})-G(x)].
\]
The limiting evolution will then be given by the equation
\begin{equation}
\label{eqprefattgen}
\dot x=f(x)+\al \la (b,x)\frac{\pa G}{\pa x_1}
+ (1-\al)\la (b,x)\sum_{k=1}^{\infty} k x_k\left[\frac{\pa G}{\pa x_{k+1}}-\frac{\pa G}{\pa x_k}\right],
\end{equation}
where $f(x)$ is obtained from the limit of \eqref{eqdefgenmeanfpool2modrep} or \eqref{eqSmolgencoagfrag}.
A strategically enhanced preferential attachment model on the evolutionary background will thus be described,
in the dynamic law of large number limit, by the controlled  infinite-dimensional ODEs \eqref{eqprefattgen}
(via discrete or continuous-time choice of parameter $b$ by the principal).

As we mentioned, a rigorous proof of the convergence is beyond the scope of this paper.
Apart from sorting out this problem, an important issue is to understand the
controllability of the limiting (now in the sense $t\to \infty$) stationary solutions,
which may lead to the possibility to develop tools for influencing the power tails of distributions (Zipf's law)
appearing in many situations of practical interest, as well as the proliferation or extinction
of certain desirable (or undesirable) characteristics of the processes of evolution.

%\appendix

\section{Appendix}
%\section{Sensitivity of ODEs in Banach spaces}

\subsection{Notations for functional spaces and measures}
\label{secnotspace}

Notations introduced here are used in the main text systematically without further reminder.

For a metric space $Z$ with a metric $\rho$, let $C(Z)$ denote the space of bounded continuous functions
equipped with the sup-norm: $\|f\|=\sup_x |f(x)|$, $C_{bLip}(Z)$ the subspace of bounded Lipschitz functions
with the norm
 \begin{equation}
\label{eqdefLipnorm}
 \|f\|_{bLip}=\|f\|+\|f\|_{Lip}, \quad  \|f\|_{Lip}=\sup_{x\neq y} \frac{|f(x)-f(y)|}{\rho(x,y)}.
\end{equation}
We may write shortly $C^k$ or $C_{bLip}$ if it is clear which $Z$ we are working with.

Since we often interpret our vectors as measures, for Euclidean space $Z$, it is convenient to use the $l_1$-norm
$|x|_1=\sum_j |x_j|$ for vectors $x\in Z$, so that for functions on $\R^n$ we define
 \begin{equation}
\label{eqdefLipnorml1}
\|f\|_{Lip} = \sup_{x\neq y} \frac{|f(x)-f(y)|}{|x-y|_1}
=\sup_j \sup \frac{|f(x)-f(y)|}{|x_j-y_j|},
\end{equation}
where the last $\sup$ is the supremum over the pairs $x,y$ that differ only in its $j$th coordinate.

For $Z$ a closed convex subset of $\R^n$, let $C^k(Z)$ denote the space of $k$ times continuously differentiable
 functions on $Z$ with uniformly bounded derivatives equipped with the norm
\[
\|f\|_{C^k(Z)}=\|f\|+\sum_{j=1}^k \| f^{(j)}\|,
\]
where $\|f^{(j)}\|$ is the supremum of the magnitudes of all partial derivatives of $f$ of order $j$.
In particular, for a differentiable function, $\|f\|_{C^1}=\|f\|_{bLip}$.

For $Z$ a closed convex subset of a Banach space $B$,
the directional derivative of a real function $F$ on $Z$ at $x$ in the direction $\xi\in Z-x$ is defined as
\begin{equation}
\label{eqdefdirder}
D_{\xi}F(x)=DF(x)[\xi]=\lim_{h\to 0_+}\frac{F(x+h\xi)-F(x)}{h},
\end{equation}
and higher order derivatives are defined recursively, for instance the second derivative is
\[
D^2F(x)[\xi,\eta]=D\left( DF(x)[\xi]\right) [\eta], \quad  \xi, \eta \in Z-x.
\]

The spaces $C^k(Z)$, $k\in \N$ of continuously differentiable functions
 are the subsets of functions from $C(Z)$ with the derivatives of order up to $k$ well defined and
continuous with respect to all their variables and having finite norms
\[
\| F\|_{C^k(Z)}=\|F\|+\sum_{l=1}^k \sup_{x\in Z}\sup_{\xi_j: \|\xi_j\|=1} |D^lF(x)[\xi_1, \cdots , \xi_l]|,
\]

Similarly the differentiability of the Banach-space-valued functionals $F:Z \to B_1$ and the corresponding spaces
$C(Z;B_1)$, $C^k(Z;B_1)$ are defined for any other Banach space $B_1$.

For instance, if $B=l^1$, then
\begin{equation}
\label{eq1defnorsmoothl1}
\| F\|_{C^1(Z)}=\|F\|+\sup_{x\in Z} \sup_k \left|\frac{\pa F}{\pa x_k}\right|,
\end{equation}
and
\begin{equation}
\label{eq2defnorsmoothl1}
\| F\|_{C^2(Z)}=\| F\|_{C^1(Z)} +\sup_{x\in Z} \sup_{k,l} \left|\frac{\pa ^2F}{\pa x_k \pa x_l}\right|.
\end{equation}

For a locally compact metric space $Z$ we denote by $\MC(Z)$ (resp. $\MC^+(Z)$)
the Banach space of signed finite Borel measures on $Z$ (resp. its subset of non-negative measures),
by $\MC_{\la}(Z)$ the ball of radius $\la$ there, with $\MC^+_{\la}(Z)=\MC_{\la}(Z)\cap \MC^+(Z)$.
According to the Riesz-Markov Theorem, the Banach space $\MC(Z)$ is the Banach dual to
the space $C_{\infty}(Z)$, which is the subspace of functions from $C(Z)$ vanishing at infinity.

For a function $f$ on $Z$ and a measure $\mu$ (not necessarily bounded) we
use the scalar-product notations $(f,\mu)=\int f(z) \mu(dz)$
for the natural pairing, whenever it is well defined.

By the celebrated Kantorovich theorem, the weak topology on $\MC^+(Z)$ can be metricized via the duality
relation with the space $C_{bLip}(Z)$, that is, via the metric
\[
d_{bLip*}(\mu,\nu)=\|\mu-\nu\|_{bLip*}=\sup_{f: \|f\|_{bLip}\le 1} \int_Z f(z) (\mu- \nu) (dz).
\]

For a closed convex subset $S$ of $\MC^+(Z)$ we shall denote by $C_{weak}(S)$ the closed subset
of $C(S)$ consisting of weakly continuous functions. We shall denote by $C_{weak}^{bLip}(S)$
the space of weakly Lipschitz functions $F$ on $S$ (which are Lipschitz with respect to $d_{bLip*}$).
We shall denote $\|F\|_{weakLip}$ the corresponding Lipschitz constant and
 $\|F\|_{weakbLip}= \|F\|+\|F\|_{weakLip}$ the norm in $C_{weak}^{bLip}(S)$.

\begin{remark}
Linguistically counterintuitive, the weak continuity is a stronger requirement than just continuity.
For any bounded  measurable $\phi$, the linear functional $F(\mu)=(\phi, \mu)=\int \phi(z) \mu(dz)$
on $\MC(Z)$ is continuous and continuously differentiable of all orders in the norm topology
 with $DF(x)[\xi]=(\phi, \xi)$, $D^2F(x)=0$. On the other hand, this $F(\mu)$ is weakly continuous
 only  if $\phi$ is continuous and weakly-$*$
continuous if additionally $\phi(z)\to 0$ for $z \to \infty$. It is weakly Lipschitz, if $\phi \in C_{bLip}(Z)$.
Only for discrete countable $Z$, the linear functionals on the space $\MC(Z)=l^1$ are continuous in the norm
if and only if they are weakly continuous. This often allows one to avoid using weak topology for $l^1$.
\end{remark}

We recall for reference the following simple
and standard general formula for the comparison of arbitrary
operator semigroups $U_N$ and $U$ with generators $L_N$ and $L$ respectively:
\begin{equation}
\label{eqcompsem1}
U_N^{T-t} g-U^{T-t}=U_N^{s-t}U^{T-s}|_{s=t}^T=\int_t^T
 U_N^{s-t} (L_N-L)U^{T-s} \, ds.
 \end{equation}
When $U_N$ as a contraction in a space of bounded functions, it implies
\begin{equation}
\label{eqcompsem2}
\|U_N^{T-t} g-U^{T-t}g\|\le (T-t) \sup_{s\in [t,T]}\|(L_N-L)U^{T-s} g\|.
\end{equation}

\subsection{Sensitivity of ODEs in Banach spaces}
\label{secdifeqinl1}

Here we put together, in a concise way, certain basic facts on the sensitivity of
ODEs in Banach space with an unbounded (in particular quadratic)  r.h.s.,
the main example of interest for us being the Banach space $l^1$ and the evolutions satisfying \eqref{eqLyapcondLupper}.

Let $B$ be a Banach space equipped with the norm $\|.\|_B$ and $B_+$ its certain convex cone.
We shall write shortly $\|.\|$ for $\|.\|_B$ when no confusion arises.
Let $B(R)$ denote the ball of radius $R$ in $B$ centered at the origin and $B_+(R)=B_+\cap B(R)$.
For a linear operator $A: B\to B$ we denote by $\|A\|_{B\to B}$ its operator norm.

Let us consider an ordinary differential equation (ODE) $\dot x =f(x)$ in $B$
with a locally Lipschitz, but generally unbounded $f$
such that for any $x\in B_+(R)$ the global solution $X(t,x)$ is uniquely defined with
\begin{equation}
\label{eqLyapcondLuppergen}
X(t,x)\in B_+(e^{at}(\|x_0\|+bt))
\end{equation}
for some constants $a,b$. Lemma \ref{lemonwellposedinl1} below motivates the use of condition \eqref{eqLyapcondLuppergen}.

Under \eqref{eqLyapcondLuppergen}, the linear operators $U^t$:
\begin{equation}
\label{eqdefflowonkineq}
U^tF(x)=F(X(t,x)), \quad t\ge 0,
\end{equation}
are well defined contractions in $C(B_+)$ forming a semigroup. In case $a=b=0$,
the operators $U^t$ form a semigroup of contractions also in $C(B_+(R))$ for any $R$.

\begin{lemma}
\label{lemonsenseBanach1}
Under \eqref{eqLyapcondLuppergen} assume additionally
that $f$ is twice continuously differentiable as a mapping on $B_+$ such that
for any $R$ and all $x\in  B_+(R)$,
\begin{equation}
\label{eq1lemonsenseBanach1}
\|f\|_{C^1(B_+(R);B)}\le D_1(R),
\quad \|f\|_{C^2(B_+(R);B)} \le D_2(R),
\end{equation}
with some continuous functions $D_1(R), D_2(R)$.
Then the solutions to $\dot x =f(x)$ are twice continuously differentiable with respect to initial data and
\[
\|X(t,.)\|_{C^1(B_+(R);B)}
\le \exp \left\{t D_1(e^{at}(R+bt))\right\},
\]
\begin{equation}
\label{eq2lemonsenseBanach1}
 \|X(t,.)\|_{C^2(B_+(R);B)}
\le t D_2(e^{at}(R+bt))\exp \left\{3t D_1(e^{at}(R+bt))\right\}.
\end{equation}
Moreover,
\begin{equation}
\label{eq3lemonsenseBanach1}
\|U^tF\|_{C^1(B_+(R))} \le  \exp \left\{t D_1(e^{at}(R+bt))\right\}
 \|F \|_{C^1(B_+(e^{at}(R+bt)))},
 \end{equation}
  \begin{equation}
\label{eq4lemonsenseBanach1}
\|U^tF\|_{C^2(B_+(R))} \le
(1+ tD_2(e^{at}(R+bt)))\exp \left\{3t D_1(e^{at}(R+bt))\right\})\|F \|_{C^2(B_+(e^{at}(R+bt)))}.
\end{equation}
\end{lemma}

\begin{proof}
Differentiating the equation $\dot x=f(x)$ with respect to initial conditions yields
\begin{equation}
\label{eq5alemonsenseBanach1}
\frac{d}{dt}DX(t,x)[\xi]=Df(X(t,x))[DX(t,x)[\xi]]=Df(X(t,x)) \circ  DX(t,x)[\xi]
\end{equation}
\begin{equation}
\label{eq5blemonsenseBanach1}
\frac{d}{dt}D^2X(t,x)[\xi, \eta]=D^2f(X(t,x))[DX(t,x)[\xi], DX(t,x)[\eta]]+Df(X(t,x)) [D^2X(t,x)[\xi,\eta]].
\end{equation}
Since the initial conditions to these equations are $DX(0,x)[\xi]=\xi$,  $D^2X(0,x)[\xi,\eta]=0$,
one deduces \eqref{eq2lemonsenseBanach1} from \eqref{eq1lemonsenseBanach1}.

Differentiating \eqref{eqdefflowonkineq} yields
\begin{equation}
\label{eq5lemonsenseBanach1}
 D(U^tF)(x)[\xi]=DF(X(t,x))[DX(t,x)[\xi]],
 \end{equation}
  \begin{equation}
\label{eq6lemonsenseBanach1}
 D^2(U^tF)(x)[\xi,\eta]=D^F(X(t,x))[DX(t,x)[\xi], DX(t,x)[\eta]]
 +DF(X(t,x))[D^2X(t,x)[\xi,\eta]]
\end{equation}
implying \eqref{eq3lemonsenseBanach1} and \eqref{eq4lemonsenseBanach1}.
\end{proof}

\subsection{Variational derivatives}
\label{secvarder}

We recall here some facts about variational derivatives
of the functionals on measures. As a consequence, we deduce the asymptotic formula
for the generator of our basic model.

For a function $F$ on a convex closed subset $S$ of $\MC(Z)$ with a locally compact metric space $Z$
the {\it variational derivative} $\frac{\de F(Y)}{\de Y(x)}$ is defined as the
directional derivative of $F(Y)$ in the direction $\delta_x$:
\begin{equation}
\label{ewdefvarder}
\frac{\de F(Y)}{\de Y(x)}=D_{\de_x} F(Y)
=\lim_{s \rightarrow 0_+} \frac{1}{s} (F(Y + s\delta_x) - F(Y)).
\end{equation}
The higher derivatives
$\delta^l F(Y)/\de Y(x_1)...\de Y(x_l)$ are defined inductively.

As it follows from the definition, if $\de F(Y)/\de Y(.)$ exists for
$x\in Z$ and depends continuously on $Y$ (in weak or norm topology),
then the function $F(Y+s\delta_x)$ of $s\in {\R}_+$ has a
continuous right derivative everywhere and hence is continuously
differentiable implying
 \begin{equation}
\label{eqvarder1}
 F(Y+\delta_x)-F(Y) =\int_0^1 \frac{\de F(Y+s \delta_x)}{\de Y(x)}\, ds.
 \end{equation}

We shall say that $F$ belongs to $C^k_{weak}(S)$,
$k=1,2, \ldots,$ if $\delta^lF(Y)/\de Y(x_1)
\ldots \de Y(x_l)$ exists for all $l=1,...,k$, all $x_1, \ldots, x_k \in Z^k$ and $Y \in S$,
and represents a continuous mapping of $k+1$ variables (when
measures equipped with the weak topology) uniformly bounded on the
sets of bounded $Y$. When defined on a bounded set $S$, these spaces become Banach
when equipped with the norm
\[
\|F\|_{C^k_{weak}(S)}=\sup_{x_1, \cdots , x_k} \,\,
\sup_{Y\in S} \left| \frac{\de^k F (Y)}{\de Y(x_1) \cdots \de Y(x_k)}\right|.
\]

\begin{remark}
Again counterintuitive, the weak differentiability does not imply the weak Lipschitz continuity.
For $\phi\in C(\R^n)$, the linear functional $F(\mu)=(\phi, \mu)=\int \phi(z) \mu(dz)$
on $\MC(Z)$ is weakly continuously differentiable of all orders, but it is weakly Lipschitz
 only if $\phi$ is Lipschitz, with $\|F\|_{weakLip}=\|\phi\|_{Lip}$.
\end{remark}

The following facts are basic formulas of the calculus for functionals on measures.
They are easy to deduce (the details are given in \cite{Ko10}).
\begin{lemma}
\label{lemvarder}
(i) One has the inclusion $C^1_{weak}(S)\subset C^1(S)$ and
 \begin{equation}
\label{eqvarder12}
 D_{\xi}F(Y)=\int \frac{\de F (Y)}{\de Y(x)} \xi (dx)
\end{equation}
\begin{equation}
\label{eqvarder2}
 F(Y + \xi) - F(Y) = \int_0^1 \left(\frac{\de F(Y + s \xi)}{\de Y(.)}, \xi\right)\, ds
\end{equation}
for $F\in C^1_{weak}(S)$ and $Y\in S,\xi \in S-Y$.

(ii) One has the inclusion $C^2_{weak}(S)\subset C^2(S)$ and
\begin{equation}
\label{eqvarder3}
 F(Y + \xi) - F(Y) = (\frac{\de F(Y)}{\de Y(.)},\xi)
  + \int_0^1 (1 - s) \left(\frac{\de^2 F(Y + s \xi)}{\de Y(.)\de Y(.)}, \xi
\otimes \xi\right)\, ds,
 \end{equation}
  for $F \in C^2_{weak}(S)$ and $Y\in S,\xi \in S-Y$.

(iii) If $t \mapsto \mu_t \in S$ is continuously
differentiable in the weak topology, then for any $F \in C^1_{weak}(S)$
\begin{equation}
\label{eqvarder4}
 \frac{d}{dt} F(\mu_t) = (\delta F(\mu_t; \cdot),\dot{\mu}_t).
 \end{equation}
\end{lemma}

Though the variational derivatives are well defined for
the general space $C^1(S)$ of strongly differentiable functions, they may not be continuous
and hence are not very handy to work with. The analogs of equations \eqref{eqvarder2} and
\eqref{eqvarder3} for $F\in C^1(S)$ and $F\in C^2(S)$ respectively are the formulas
\begin{equation}
\label{eqvarder2str}
 F(Y + \xi) - F(Y) = \int_0^1 D F(Y+s\xi)[\xi] \, ds,
\end{equation}
\begin{equation}
\label{eqvarder3str}
 F(Y + \xi) - F(Y) =  D_{\xi} F(Y) + \int_0^1 (1 - s) D^2 F(Y + s \xi)[\xi,\xi]\, ds,
 \end{equation}
 valid for $\xi \in S-x$.

These rules extend to measure-valued
functions on $\MC(Z)$. Namely, a mapping $\Phi:\MC^+(Z)
\mapsto \MC(Z')$ with another set $Z'$ has a {\it weak variational derivative}
$\delta \Phi /\de Y(x)$, if for any $Y\in \MC^+(Z)$, $x\in Z$ the limit
\[
\frac{\delta \Phi}{\delta Y(x)}=\lim_{s\to 0_+} \frac{1}{s} (\Phi
(Y+s\delta_x)-\Phi (Y))
\]
exists in the weak topology of $\MC(Z')$ and is a finite signed
measure on $Z'$. Higher derivative are defined inductively. We shall
say that $\Phi$ belongs to $C^l_{weak}(\MC(Z);\MC(Z'))$, $l = 1,2, \ldots,$
if  the weak variational derivatives $\delta^k
\Phi(Y;x_1, \ldots, x_k)$ exist for all $k=1,...,l$, all $x_1, \ldots, x_k \in Z^k$ and
$Y \in \MC(Z)$, and represent continuous in the sense of the weak
topology mapping $\MC(Z)\times Z^k \mapsto \MC(Z')$, which is
bounded on the bounded subsets of $Y$.

\begin{remark}
Unlike real functions, the inclusion $C^l_{weak}(\MC(Z);\MC(Z))\subset C^l(\MC(Z);\MC(Z))$
does not hold anymore. For instance, if $Z$ is a one-point set, we have the opposite inclusion
$C^l(\R;\MC(Z))\subset C^l_{weak}(\R;\MC(Z))$.
\end{remark}

The following chain rule is straightforward (details of the proof see e. g. \cite{Ko10}).

\begin{lemma}
\label{lemvardirchainrule}
(i) Let $\Phi \in C^1_{weak}(\MC(Z);\MC(Z))$ and
$F\in C^1_{weak}(\MC(Z))$, then the composition $F\circ \Phi (Y)=F(\Phi (Y))$
 belongs to $C^1_{weak}(\MC(Z))$ and
\begin{equation}
\label{eqlemvardirchainrule}
 \frac{\delta F}{\delta Y(x)}(\Phi(Y))=\int_Z
\frac{\delta F (W)}{\delta W(y)}\mid_{W=\Phi (Y)}\frac{\delta
\Phi}{\delta Y(x)}(Y,dy).
\end{equation}

(ii) Similarly, if $\Phi \in C^1(\MC(Z);\MC(Z))$ and
$F\in C^1(\MC(Z))$, then the composition $F\circ \Phi (Y)=F(\Phi (Y))$
 belongs to $C^1(\MC(Z))$ and
 \begin{equation}
\label{eqlemvardirchainrulestr}
D_{\xi} (F\circ \Phi) (Y) =DF (\Phi (Y))[D_{\xi} \Phi (Y)],
\end{equation}
for any $\xi$. This turns to \eqref{eqlemvardirchainrule} for $\xi=\de_x$.
\end{lemma}

The following technical result is the key ingredient in the proof of Theorem \ref{th2}.

\begin{lemma}
\label{lembasas}
Let a measurable function $R(b,\mu)$ on $\R\times \MC(Z)$ be given.
For a pair of different points $z_1,z_2$ of $Z$ and a measure $\mu\in \MC(Z)$,
let $z_l(\mu), z_s(\mu)$ (with $l$ standing for 'large' and $s$ for 'small')
denote the same pair, but ordered in such a way that $R(z_l,\mu)\ge R(z_s, \mu)$
(if the values are equal, the choice of ordering is irrelevant). Let
\begin{equation}
\label{eqdefcontwithmajorN0}
L_Nf (\de_{\x}/N)=\frac{\ka}{N} \sum_{(i,j)}
[R(x_j,\de_{\x}/N)-R(x_i,\de_{\x}/N)]
[f(\de_{\x}/N-\de_{x_i}/N+\de_{x_j}/N)-f(x)]
\end{equation}
where $\x=(x_1, \cdots, x_N)$ and the sum is over all pairs $(i,j)$ of distinct indices ordered in such a way that
$R(x_j,\de_{\x}/N)>R(x_i,\de_{\x}/N)$
(the order is irrelevant if the corresponding values of $R$ coincide).

Then, for  $f\in C^2_{weak}(\MC(Z))$,
\[
L_Nf(\mu)
=\ka \int_Z \int_Z \frac{\de f(\mu)}{\de \mu (z_2)}[R(z_2,\mu)-R(z_1,\mu)]\mu(dz_1)\mu(dz_2)
\]
\[
+\frac{\ka}{2N} \int_0^1 (1-s) \int_K \int_K \mu(dz_1)\mu(dz_2) \, ds [R(z_l(\mu),\mu)-R(z_s(\mu),\mu)]
\]
\begin{equation}
\label{eq1lembasas}
\times \left(\frac{\de^2 f}{\de \mu (z_2)\de \mu(z_2)}
-2\frac{\de^2 f}{\de \mu (z_2)\de \mu(z_1)}+\frac{\de^2 f}{\de \mu (z_1)\de \mu(z_1)}\right)
\left(\mu+\frac{s}{N}(\de_{z_l(\mu)}-\de_{z_s(\mu)})\right)
\end{equation}
with $\mu=\de_{\x}/N$.
\end{lemma}

\begin{proof}
Applying \eqref{eqvarder3} one gets
\[
L_Nf(\mu)
=\frac{\ka}{N} \sum_{i,j: R(x_j,\mu)>R(x_i,\mu)}[R(x_j,\mu)-R(x_i,\mu)]
\]
\[
\times
\left[
\left(\frac{\de f (\mu)}{\de \mu(.)}, \frac{\de_{x_j}-\de_{x_i}}{N}\right) +\int_0^1(1-s)
\left(\frac{\de^2 f (\mu+(\de_{x_j}-\de_{x_i})/N)}{\de \mu(.)\de \mu (.)}, \frac{(\de_{x_j}-\de_{x_i})^{\otimes 2}}{N^2}\right)
 \, ds \right],
\]
or equivalently
\[
L_Nf(\mu)
=\frac{\ka}{N^2} \sum_{I=\{i,j\}}
[R(x_j,\mu)-R(x_i,\mu)]
\left(\frac{\de f (\mu)}{\de \mu(x_j)}-\frac{\de f (\mu)}{\de \mu(x_i)}\right)
 \]
 \[
 +\frac{\ka}{N^3} \sum_{I=\{i,j\}} \int_0^1(1-s) ds [R(x_l,\mu)-R(x_s,\mu)]
 \]
 \[
 \times
\left(\frac{\de^2 f}{\de \mu(x_i)\de \mu (x_i)}-2\frac{\de^2 f}{\de \mu(x_i)\de \mu (x_j)}+\frac{\de^2 f}{\de \mu(x_j)\de \mu (x_j)}\right)
 \left(\mu+\frac{s}{N}(\de_{x_l(\mu)}-\de_{x_s(\mu)})\right),
\]
where the summation is over the two-point subsets $I$ of $\{1, \cdots , N\}$.
It is seen directly that this rewrites as \eqref{eq1lembasas}.
\end{proof}

\subsection{On measure-valued ODEs with the Lyapunov condition}
\label{secdifeqinl3}

Let $Z$ be a locally compact space.
Here we recall the basic facts on the growth of
positivity preserving ordinary differential equations (ODEs) in $\MC(Z)$
with an unbounded r.h.s. satisfying the Lyapunov condition.

Let us consider again an ODE $\dot x=f(x)$ in $\MC(Z)$
with a continuous, but generally unbounded $f$.
We are interested here in evolutions preserving positivity, that is, such that for
 any initial $x \in \MC^+(Z)$ the solution $x(t)$ belongs to $\MC^+(Z)$ for all $t$.
 This implies that $f$ must be conditionally positive, in the sense that for any $x \in \MC^+(Z)$,
 the negative part of $f(x)$ is absolutely continuous with respect to $f(x)$. In case $\MC^+(Z)=l^1_+$ this means that
for any $x\in l^1_+$ with $x_k=0$ one has $f_k(x)\ge 0$.

 \begin{remark}
By Theorem 6.21 of \cite{Ko10}, conditionally positive bounded $f$ have the following structure: there exist
a family of stochastic kernels $\nu(x,y,dz)$ in $Z$, $x\in \MC(Z)$, and a non-negative function $a(x,z)$
on $\MC(Z) \times Z$ such that
\begin{equation}
\label{eqdefpositpreseqden}
f(x)(dy) =\int_Z x (dz) \nu(x,z,dy) -a(x,y) x (dy).
\end{equation}
In particular, if $Z=\N$, this means the existence of nonnegative
functions  $\nu (j,x,k)$ and $a(j,x)$ on $\N \times l^1 \times \N$ and on $\N \times l^1$ respectively such that
\begin{equation}
\label{eqdefpositpreseqden1}
f_k(x) =\sum_j x_j \nu(x,j,k) -a(x,k) x_k.
\end{equation}
\end{remark}

A continuous function $L$ on $Z$, bounded below by a positive constant,
will be referred to as a Lyapunov function or a barrier. For any such function,
let us define the subset $\MC(Z,L)$ of $\MC(Z)$ of measures $x$ such that
\begin{equation}
\label{eqdefweighnorml1}
\|x\|_L=\int L(z) |x|(dz)=(|x|,L) <\infty,
\end{equation}
which is itself a Banach space with the norm $\|.\|_L$. Let us denote by $B(L,R)$ the ball in $\MC(Z,L)$
of radius $R$ and let $\MC(Z,L)_+=\MC(Z,L)\cap \MC^+(Z)$, $B_+(L,R)=B(L,R)\cap \MC^+(Z)$.
For the case $Z=\N$ let us write $l^1(L)$ for $\MC(Z,L)$. In particular, $l^1(\1)=l^1=\MC(\Z)$, where $\1$
denotes of course the function that equals $1$ everywhere.

Let us say that the equation $\dot x=f(x)$ and the function $f(x)$ are $L$-subcritical
(respectively,  satisfy the Lyapunov condition for $L$) if
 $f:\MC(Z,L)_+ \to \MC(Z,L)$ and
\begin{equation}
\label{eq1lemonwellposedinl1}
(L, f(x))=\int L(z)f(x)(dz) \le 0
\end{equation}
( respectively
\begin{equation}
\label{eqLyapcondL}
(L, f(x))\le a (L,x)+b
\end{equation}
for all $x\in \MC(Z,L)_+$ and some constants $a,b$).

\begin{lemma}
\label{lemonwellposedinl1}

(i) Suppose the function $f$ is conditionally positive, satisfies the Lyapunov condition for
a Lyapunov function $L$ on $Z$ and
is Lipschitz either weakly or in the norm of $\MC(Z,L)$ or $\MC(Z)$
on any bounded subset of $\MC(Z,L)_+$. Then, for any $x\in \MC(Z,L)_+$,
the Cauchy problem of equation $\dot x=f(x)$ with
initial condition $x$ at time $s\ge 0$ has
a unique global (that is defined for all times) solution
$X(t,x)$ in $\MC(Z,L)_+$ with derivative understood with respect to the corresponding
 topology. Moreover,
\begin{equation}
\label{eqLyapcondLupper}
X(t,x)\in B_+(L, e^{at}(\|x_0\|_L+bt)).
\end{equation}
In particular, any ball $B_+(L,R)$ is invariant under an $L$-subcritical evolution.

(ii) If additionally to \eqref{eqLyapcondL}, one has
\begin{equation}
\label{eqLyapcondLlower}
(L, f(x))\ge -a_1 (L,x)
\end{equation}
with a constant $a_1$, then
\begin{equation}
\label{eqLyapcondLlower1}
(L,X(t,x))\ge e^{-a_1 t}(L,x).
\end{equation}

(iii) Finally, if instead of \eqref{eqLyapcondL} one has
\begin{equation}
\label{eqLyapcondLlowerup}
(L, f(x))=a (L,x)+b,
\end{equation}
then
\begin{equation}
\label{eqLyapcondLlowerup1}
(L, X(t,x)))=e^{a t}[(L,x)+\frac{b}{a} (1-e^{-at})].
\end{equation}
\end{lemma}

\begin{proof}
(i) By local Lipschitz continuity and conditional positivity, equation $\dot x=f(x)$
is locally well-posed and preserves positivity.
Moreover, by the Lyapunov condition
\[
 (L,x(t))\le (L,x)+a\int_0^t (L, x(s)) \, ds +bt,
 \]
 so that by Gronwall's lemma (and the preservation of positivity)
\[
0\le  (L,x(t)) \le e^{at} [(L,x)+bt]
\]
implying that the solution can be extended to all times with required bounds.

(ii) This is clear, as \eqref{eqLyapcondL} implies
\[
\frac{d}{dt} (L,x(t)) \ge -a_1 (L,x(t)).
\]

(iii) Equation \eqref{eqLyapcondLlowerup} implies
\[
\frac{d}{dt} (L,x(t))=a(L,x(t))+b,
\]
leading to \eqref{eqLyapcondLlowerup1}.

\end{proof}

{\bf Acknowledgements}. I am grateful to Alain Bensoussan, Mark Kilgour, Oleg Malafeyev
and Didier Sornette for very useful comments to the initial drafts
of this manuscript.

\end{document}